\documentclass[12pt]{article}
\usepackage[a4paper,margin=1in]{geometry}

\usepackage{amssymb,amsthm,amsmath,mathtools}
\mathtoolsset{showonlyrefs=true}

\newtheorem{theorem}{Theorem}[section]
\newtheorem{lemma}[theorem]{Lemma}
\newtheorem{corollary}[theorem]{Corollary}

\newtheorem{fact}[theorem]{Fact}
\newtheorem{proposition}[theorem]{Proposition}
\newtheorem{problem}[theorem]{Problem}

\theoremstyle{definition}
\newtheorem{definition}[theorem]{Definition}

\theoremstyle{remark}
\newtheorem{remark}[theorem]{Remark}

\usepackage{hyperref}
\usepackage[style=alphabetic,isbn=false,url=false,doi=false,giveninits=true,maxnames=9,sorting=nyt,backend=bibtex]{biblatex} 
\renewbibmacro{in:}{}
\bibliography{reference.bib}
\DeclareFieldFormat[article,periodical]{volume}{\mkbibbold{#1}}

\usepackage{enumitem}
\setlist[enumerate,1]{label=(\arabic*)}
\setlist[enumerate,2]{label=(\roman*)}

\numberwithin{equation}{section}
\numberwithin{figure}{section}
\numberwithin{table}{section}

\usepackage{multirow}

\DeclareMathOperator{\supp}{supp}
\DeclareMathOperator{\im}{Im}
\DeclareMathOperator{\re}{Re}

\title{Continuous solutions of the complex Kac--Bernstein functional equation on the integers and the real numbers}
\author{Takashi Satomi}

\begin{document}

\maketitle

\begin{abstract}

In this paper, the continuous solutions of the complex Kac--Bernstein functional equation
\begin{align}
  & \phantom{ {} = {} }  f_1 ( u + v ) f_2 ( u - v ) f_1 ( u' + v' ) f_2 ( u' - v' ) \\
  & = f_1 ( u + v' ) f_2 ( u - v' ) f_1 ( u' + v ) f_2 ( u' - v )
\end{align}
are classified for $ \mathbb{ Z } $ and $ \mathbb{ R } $.
By using this classification for $ \mathbb{ Z } $, we consider a generalization of the Kac--Bernstein theorem on the one-dimensional torus $ \mathbb{ T } $ by Baryshnikov--Eisenberg--Stadje from probability Borel measures to complex Borel measures.
Similarly, the original Kac--Bernstein theorem on $ \mathbb{ R } $ is generalized from probability Borel measures to complex Borel measures.

\end{abstract}



\section{Introduction}

In this paper, the continuous solutions of the complex Kac--Bernstein functional equation are classified for $ \mathbb{ Z } $ (Theorem \ref{thm:Z-characterize}) and $ \mathbb{ R } $ (Corollary \ref{cor:R-characterize}).
Furthermore, by using these results, we study the Kac--Bernstein theorem for complex Borel measures on $ \mathbb{ T } \coloneqq \mathbb{ R } / 2 \pi \mathbb{ Z } $ (Corollary \ref{cor:complex-T}) and $ \mathbb{ R } $ (Corollary \ref{cor:Kac-Bernstein-complex-R}).
These results generalize previous classifications from probability measures to complex measures.
The probability-measure cases on $ \mathbb{ T } $ and $ \mathbb{ R } $ were established by Baryshnikov--Eisenberg--Stadje (Corollary \ref{cor:Baryshnikov-Eisenberg-Stadje}) and Kac--Bernstein (Corollary \ref{cor:original-Kac-Bernstein}), respectively.

For functions $ f_1 , f_2 \colon G \to \mathbb{ C } $ on a locally compact abelian group $ G $, the equation
\begin{align}
  & \phantom{ {} = {} }  f_1 ( u + v ) f_2 ( u - v ) f_1 ( u' + v' ) f_2 ( u' - v' ) \\
  & = f_1 ( u + v' ) f_2 ( u - v' ) f_1 ( u' + v ) f_2 ( u' - v ) \label{eq:complex-Kac-Bernstein}
\end{align}
is called the complex Kac--Bernstein functional equation ($ u , v , u' , v' \in G $).
If 
\begin{align}
    f_1 ( 0 ) = f_2 ( 0 ) = 1 \label{eq:probability-dual}
\end{align}
holds, then \eqref{eq:complex-Kac-Bernstein} is equivalent to the original Kac--Bernstein functional equation
\begin{align}
    f_1 ( u + v ) f_2 (u - v ) = f_1 ( u ) f_2 ( u ) f_1 ( v ) f_2 ( - v ) \label{eq:Kac-Bernstein-original}
\end{align}
(Proposition \ref{prop:original-Kac-Bernstein}).
The main theorem is a classification of the continuous solutions of \eqref{eq:complex-Kac-Bernstein} in the case of $ G = \mathbb{ Z } $ and $ G = \mathbb{ R } $.
When $ G = \mathbb{ Z } $, the following theorem holds:

\begin{theorem}
    \label{thm:Z-characterize}

The functions $ f_1 , f_2 \colon \mathbb{ Z } \to \mathbb{ C } $ satisfy \eqref{eq:complex-Kac-Bernstein} for any $ u , u' , v , v' \in \mathbb{ Z } $ if and only if at least one of the following conditions holds:

\begin{enumerate}
  \item \label{item:Z-characterize-f1-zero}
  
  One has $ f_1 \equiv 0 $.

  \item \label{item:Z-characterize-f2-zero}
  
  One has $ f_2 \equiv 0 $.

\item \label{item:Z-characterize-f1-one-point}

There exists $ a \in \mathbb{ Z } $ such that
\begin{align}
    \supp f_1 = \{ a \} , & & 
    \# ( \supp f_2 \cap ( a + 2 \mathbb{ Z } ) ) = 1,
\end{align}
where $ \supp f \coloneqq \{ u \in \mathbb{ Z } \mid f ( u ) \neq 0 \} $ and $ \# S $ denotes the number of elements of $ S $.

\item \label{item:Z-characterize-f2-one-point}

There exists $ a \in \mathbb{ Z } $ such that
\begin{align}
    \supp f_2 = \{ a \} , & &
    \# ( \supp f_1 \cap ( a + 2 \mathbb{ Z } ) ) = 1.
\end{align}

\item \label{item:Z-characterize-independent}

There exists $ a \in \mathbb{ Z } $ such that
\begin{align}
  \supp f_1 \subset a + 2 \mathbb{ Z }, & &
  \supp f_2 \subset a + 1 + 2 \mathbb{ Z }.
\end{align}

\item \label{item:Z-characterize-six-point}

There exist $ a_1 , a_2 , k \in \mathbb{ Z } $ with $ k , a_1 + a_2 \in 1 + 2 \mathbb{ Z } $ such that
\begin{align}
    \supp f_1 \subset \{ a_1 , a_1 \pm k \}, &  &
\supp f_2 \subset \{ a_2 , a_2 \pm k \} \label{eq:Z-characterize-six-point-support}
\end{align}
and additionally
\begin{align}
    f_1 ( a_1 )^2 f_2 ( a_2 - k ) f_2 ( a_2 + k ) = f_2 ( a_2 )^2 f_1 ( a_1 - k ) f_1 ( a_1 + k ). \label{eq:Z-characterize-six-point-equation}
\end{align}

\item \label{item:Z-characterize-Gauss}

There exist $c_1 , c_2 , c_3 , c_4 , c_5 , c_6 \in \mathbb{ C } $, $ k = 1 , 3 , 5 , \cdots $, and $ a_1 , a_2 \in \mathbb{ Z } $ with $ a_1 + a_2 \in 2 \mathbb{ Z } $ such that
\begin{align}
  f_1 ( u )
  & =
  \left\{
  \begin{aligned}
    & e^{ c_1 u^2 + c_2 ( - 1 )^u + c_3 u + c_4 } & & \text{ if $ u \in a_1 + k \mathbb{ Z } $} \\
    & 0 & & \text{ if $ u \notin a_1 + k \mathbb{ Z } $ }
  \end{aligned}
  \right. , \\
  f_2 ( u )
  & =
  \left\{
  \begin{aligned}
    & e^{ c_1 u^2 - c_2 ( - 1 )^u + c_5 u + c_6 } & & \text{ if $ u \in a_2 + k \mathbb{ Z } $} \\
    & 0 & & \text{ if $ u \notin a_2 + k \mathbb{ Z } $ }
  \end{aligned}
  \right. .
\end{align}

\end{enumerate}

\end{theorem}

By using Theorem \ref{thm:Z-characterize}, the Kac--Bernstein theorem (Corollary \ref{cor:Baryshnikov-Eisenberg-Stadje}) on $ \mathbb{ T } $ proved by Baryshnikov--Eisenberg--Stadje can be generalized from probability Borel measures to complex Borel measures (Corollary \ref{cor:complex-T}).
Furthermore, by using Theorem \ref{thm:Z-characterize}, the continuous solutions of \eqref{eq:complex-Kac-Bernstein} on $ G = \mathbb{ R } $ can be classified as follows:

\begin{corollary}
  \label{cor:R-characterize}

The continuous functions $ f_1 , f_2 \colon \mathbb{ R } \to \mathbb{ C } $ satisfy \eqref{eq:complex-Kac-Bernstein} for any $ u , u' , v , v' \in \mathbb{ R } $ if and only if one of the following conditions holds:

\begin{enumerate}
  \item \label{item:R-characterize-f1-zero}
  
  One has $ f_1 \equiv 0 $.

  \item \label{item:R-characterize-f2-zero}
  
  One has $ f_2 \equiv 0 $.

\item \label{item:R-characterize-Gauss}

There exist $ d_1 , d_2 , d_3 , d_4 , d_5 \in \mathbb{ C } $ such that
\begin{align}
  f_1 ( u )
  = e^{ d_1 u^2 + d_2 u + d_3 }, & &
  f_2 ( u )
  = e^{ d_1 u^2 + d_4 u + d_5 } \label{eq:R-characterize-Gauss-equation}
\end{align}
hold for any $ u \in \mathbb{ R } $.

\end{enumerate}

\end{corollary}

Corollary \ref{cor:R-characterize} gives a generalization (Corollary \ref{cor:Kac-Bernstein-complex-R}) of the Kac--Bernstein theorem (Corollary \ref{cor:original-Kac-Bernstein}) on $ \mathbb{ R } $ from probability Borel measures to complex Borel measures.

The organization of this paper is as follows.
In Section \ref{sec:main-relate}, we see that the complex Kac--Bernstein functional equation \eqref{eq:complex-Kac-Bernstein} arises naturally from a generalization of the Kac--Bernstein theorem to complex measures, and we study the Kac--Bernstein theorem on $ \mathbb{ T } $ and $ \mathbb{ R } $ by using Theorem \ref{thm:Z-characterize} and Corollary \ref{cor:R-characterize}, respectively.
In Section \ref{sec:Kac-Bernstein-probability}, we see the main theorem and the results in Section \ref{sec:main-relate} in the case of probability measures, and we compare them with known results.
In Section \ref{sec:Z-characterize-proof}, we prove Theorem \ref{thm:Z-characterize}.
In Section \ref{sec:R-characterize-proof}, we prove Corollary \ref{cor:R-characterize} by using Theorem \ref{thm:Z-characterize}.
In Section \ref{sec:in-main-relate-proof}, we give the proofs of several results stated in Sections \ref{sec:main-relate} and \ref{sec:Kac-Bernstein-probability}.

\section{Relation between the main theorem and the Kac--Bernstein theorem}
\label{sec:main-relate}

In this section, we obtain a generalization of the Kac--Bernstein theorem to complex measures by Theorem \ref{thm:Z-characterize} and Corollary \ref{cor:R-characterize}.
The proofs of the results in this section will be postponed to Section \ref{sec:in-main-relate-proof}.

The complex Kac--Bernstein functional equation \eqref{eq:complex-Kac-Bernstein} naturally arises as a condition for complex Borel measures whose sum and difference are independent (Proposition \ref{prop:Kac-Bernstein-Fourier}).
Thus, the following problem on $ X = \mathbb{ T } $ and $ X = \mathbb{ R } $ can be considered by using Theorem \ref{thm:Z-characterize} and Corollary \ref{cor:R-characterize}, respectively:

\begin{problem}
  \label{Q:Kac-Bernstein-complex}

Let $ F \colon X \times X \to X \times X $ be the map defined by $ F ( x_1 , x_2 ) \coloneqq ( x_1 + x_2 , x_1 - x_2 ) $ for a locally compact abelian group $ X $.
Classify pairs of complex Borel measures $ ( \mu_1 , \mu_2 ) $ on $ X $ whose pushforward measure $ F_* ( \mu_1 \times \mu_2 ) $ is the product measure of two complex measures on $ X $.
\end{problem}

Problem \ref{Q:Kac-Bernstein-complex} can be reformulated by using the Fourier--Stieltjes transform.
We denote the Pontryagin dual of a locally compact abelian group $ X $ by $ \hat{ X } $, and the Fourier--Stieltjes transform of a complex Borel measure $ \mu $ by $ \hat{ \mu } \colon \hat{ X } \to \mathbb{ C } $.
For instance, when $ X = \mathbb{ T } $, we have $ \hat{ X } = \mathbb{ Z } $ and
\begin{align}
    \hat{ \mu } ( u ) = \int_{ \mathbb{ T } } e^{ - i u x } d \mu ( x )
\end{align}
for any $ u \in \mathbb{ Z } $.
When $ X = \mathbb{ R } $, we have $ \hat{ X } = \mathbb{ R } $ and 
\begin{align}
    \hat{ \mu } ( u ) = \int_{ \mathbb{ R } } e^{ - 2 \pi i u x } d \mu ( x )
\end{align}
for any $ u \in \mathbb{ R } $.

The Fourier--Stieltjes transform of
$ F_* ( \mu_1 \times \mu_2 ) $ at $ ( u , v ) $
is given by
$ \hat{ \mu_1 } ( u + v ) \hat{ \mu_2 } ( u - v ) $ (Lemma \ref{lem:pushforward-Fourier}).
Thus, the condition of Problem \ref{Q:Kac-Bernstein-complex} holds if and only if $ \hat{ \mu_1 } ( u + v ) \hat{ \mu_2 } ( u - v ) $ can be factored into a function of $ u $ and one of $ v $.
This leads to the following proposition:

\begin{proposition}
\label{prop:Kac-Bernstein-Fourier}

Let $ X $ and $ F $ be as in Problem \ref{Q:Kac-Bernstein-complex}.
Then the following conditions on a pair $ ( \mu_1 , \mu_2 ) $ of complex Borel measures on $ X $ are equivalent:

\begin{enumerate}
  \item \label{item:Kac-Bernstein-Fourier-independent}

  The pair $ ( \mu_1 , \mu_2 ) $ satisfies the condition of Problem \ref{Q:Kac-Bernstein-complex}, that is, the pushforward measure $ F_* ( \mu_1 \times \mu_2 ) $ is the product measure of two complex measures.

  \item \label{item:Kac-Bernstein-Fourier-equation}

  The functions $ f_1 , f_2 \colon \hat{ X } \to \mathbb{ C } $ defined as
\begin{align}
    f_1 \coloneqq \hat{ \mu_1 }, & & f_2 \coloneqq \hat{ \mu_2 } \label{eq:f1-f2-Fourier}
\end{align}
satisfy \eqref{eq:complex-Kac-Bernstein} for any $ u , u' , v , v' \in \hat{ X } $.

\end{enumerate}
  
\end{proposition}

Let us consider the case of $ X = \mathbb{ T } $.
If $ \mu_1 $ and $ \mu_2 $ satisfy the equivalent conditions of Proposition \ref{prop:Kac-Bernstein-Fourier}, then, by Theorem \ref{thm:Z-characterize}, $ f_1 $ and $ f_2 $ satisfy at least one of the conditions of that theorem.
Thus, Theorem \ref{thm:Z-characterize} is useful to consider Problem \ref{Q:Kac-Bernstein-complex}.
For this purpose, we introduce some notation.

\begin{definition}

Let $ X $ be a locally compact abelian group.

\begin{enumerate}
\item

The Dirac measure of $ x \in X $ is written as $ \delta_x $.

\item

When $ X $ is compact, the normalized Haar measure on $ X $ is written as $ dm_X $.

\item 

When $ X = \mathbb{ T } $, we denote by $ M_+ $ and $ M_- $ the sets of all complex Borel measures $ \mu $ on $ \mathbb{T} $ satisfying $ \mu * \delta_\pi = \mu $ and $ \mu * \delta_\pi = - \mu $, respectively.

\item

The upper half plane is written as $ \mathbb{ H } \coloneqq \{ \tau \in \mathbb{ C } \mid \im \tau > 0 \} $.
The Jacobi theta function $ \vartheta \colon \mathbb{ C } \times \mathbb{ H } \to \mathbb{ C } $ is defined as
\begin{align}
    \vartheta ( z , \tau ) \coloneqq \sum_{ u = - \infty }^{ \infty } e^{ \pi i u^2 \tau + 2 \pi i u z }
\end{align}
for $ z \in \mathbb{ C } $ and $ \tau \in \mathbb{ H } $.

\end{enumerate}
    
\end{definition}

In the case of $ X = \mathbb{ T } $, the measures which satisfy the condition of Problem \ref{Q:Kac-Bernstein-complex} can be classified by using Theorem \ref{thm:Z-characterize} and Proposition \ref{prop:Kac-Bernstein-Fourier} as follows:

\begin{corollary}
      \label{cor:complex-T}

A pair $ ( \mu_1 , \mu_2 ) $ of complex Borel measures on $ X = \mathbb{ T } $ satisfies the condition of Problem \ref{Q:Kac-Bernstein-complex} if and only if at least one of the following conditions holds:

\begin{enumerate}
  \item \label{item:complex-T-f1-zero}
  
  One has $ \mu_1 = 0 $.

  \item \label{item:complex-T-f2-zero}
  
  One has $ \mu_2 = 0 $.

\item \label{item:complex-T-f1-one-point}

There exist $ a_1 , a_2 \in \mathbb{ Z } $ with $ a_1 + a_2 \in 2 \mathbb{ Z } $, $ \nu \in M_- $, and $ A_1 , A_2 \in \mathbb{ C } $ such that
\begin{align}
    \mu_1 & = A_1 e^{ i a_1 x } dm_{ \mathbb{ T } } (x), & \mu_2 = e^{ i a_2 x } ( A_2 dm_{ \mathbb{ T } } ( x ) + \nu ).
\end{align}

\item \label{item:complex-T-f2-one-point}

There exist $ a_1 , a_2 \in \mathbb{ Z } $ with $ a_1 + a_2 \in 2 \mathbb{ Z } $, $ \nu \in M_- $, and $ A_1 , A_2 \in \mathbb{ C } $ such that
\begin{align}
    \mu_1 & = e^{ i a_1 x } ( A_1 dm_{ \mathbb{ T } } (x) + \nu ), & \mu_2 = A_2 e^{ i a_2 x } dm_{ \mathbb{ T } } ( x ).
\end{align}

\item \label{item:complex-T-independent-1-even}
One has $ \mu_1 \in M_+ $ and $ \mu_2 \in M_- $.

\item \label{item:complex-T-independent-1-odd}
One has $ \mu_1 \in M_- $ and $ \mu_2 \in M_+ $.

\item \label{item:complex-T-six-point}

There exist $ a_1 , a_2 , k \in \mathbb{ Z } $ with $ k , a_1 + a_2 \in 1 + 2 \mathbb{ Z } $ and $ A_1 , A_2 , A_3 , A_4 , A_5 , A_6 \in \mathbb{ C } $ with $ {A_1}^2 A_5 A_6 = {A_4}^2 A_2 A_3 $ such that
\begin{align}
    \mu_1 & = ( A_1 e^{ i a_1 x } + A_2 e^{ i ( a_1 - k ) x } + A_3 e^{ i ( a_1 + k ) x } ) dm_{ \mathbb{ T } }  ( x ), \\
    \mu_2 & = ( A_4 e^{ i  a_2 x } + A_5 e^{ i ( a_2 - k ) x } + A_6 e^{ i ( a_2 + k ) x } ) dm_{ \mathbb{ T } }  ( x ).
\end{align}

\item \label{item:complex-T-Gauss}

There exist $ c_2 , c_3', c_4 , c_5', c_6 \in \mathbb{ C } $, $ \tau \in \mathbb{ H } $, $ k = 1 , 3 , 5 , \cdots $, and $ a_1 , a_2 \in \mathbb{ Z } $ with $ a_1 + a_2 \in 2 \mathbb{ Z } $ such that
\begin{align}
\mu_1 & = e^{ i a_1 x + c_4 } \left( \cosh ( c_2 ) \vartheta \left( \frac{ k x }{ 2 \pi } + c_3' , \tau \right) + \sinh ( c_2 ) \vartheta \left( \frac{ k x }{ 2 \pi } + c_3' + \frac{ 1 }{ 2 } , \tau \right) \right) dm_{ \mathbb{ T } } ( x ), \\
\mu_2 & = e^{ i a_2 x + c_6 } \left( \cosh ( c_2 ) \vartheta \left( \frac{ k x }{ 2 \pi } + c_5' , \tau \right) - \sinh ( c_2 ) \vartheta \left( \frac{ k x }{ 2 \pi } + c_5' + \frac{ 1 }{ 2 } , \tau \right) \right) dm_{ \mathbb{ T } } ( x ).
\end{align}

\item \label{item:complex-T-boundary}

There exist $ C \in \mathbb{ R } $ and a complex Borel measure $ \rho $ on $ \mathbb{ T } $ such that $ \hat{ \rho } ( u ) = e^{ i C u^2 } $ for any $ u \in \mathbb{ Z } $, and moreover there exist $ c_2 , c_4 , c_6 \in \mathbb{ C } $, $ \theta_1 , \theta_2 \in \mathbb{ T } $, $ k = 1 , 3 , 5 , \cdots $, and $ a_1 , a_2 \in \mathbb{ Z } $ with $ a_1 + a_2 \in 2 \mathbb{ Z } $ such that
\begin{align}
\mu_1 & = e^{ i a_1 x + c_4 } \left( \cosh ( c_2 ) \rho ( k x + \theta_1 ) + \sinh ( c_2 ) \rho ( k x + \theta_1 + \pi ) \right), \\
\mu_2 & = e^{ i a_2 x + c_6 } \left( \cosh ( c_2 ) \rho ( k x + \theta_2 ) - \sinh ( c_2 ) \rho ( k x + \theta_2 + \pi ) \right)
\end{align}
hold.
Here, the measure $ \rho ( k x + \theta ) $ is defined as
\begin{align}
    \int_{ \mathbb{ T } } \phi ( x ) d \rho ( k x + \theta )
    \coloneqq \frac{ 1 }{ k } \int_{ \mathbb{ T } } \sum_{ x \in \mathbb{ T }, \; y = k x + \theta } \phi ( x ) d \rho ( y ). \label{eq:complex-T-boundary-scaling}
\end{align}

\end{enumerate}

\end{corollary}

\begin{remark}

In the case of Corollary \ref{cor:complex-T} \ref{item:complex-T-boundary}, there exists $ \rho $ satisfying the condition for any $ C \in \mathbb{ Q } \pi $ because $ e^{ i C u^2 } $ is periodic.
On the other hand, when $ C \notin \mathbb{ Q } \pi $, the author does not know whether there exists $ \rho $ satisfying the condition.

\end{remark}

By Theorem \ref{thm:Z-characterize} and Proposition \ref{prop:Kac-Bernstein-Fourier}, the Fourier--Stieltjes transforms $ f_1 $ and $ f_2 $ of $ \mu_1 $ and $ \mu_2 $, respectively, in Corollary \ref{cor:complex-T} satisfy at least one of the conditions of Theorem \ref{thm:Z-characterize}.
More precisely, the conditions of Theorem \ref{thm:Z-characterize} for $ f_1 $ and $ f_2 $ correspond to the conditions of Corollary \ref{cor:complex-T}, as shown in Table \ref{tab:Z-characterize-complex-T-correspond}.
Corollaries \ref{cor:original-Kac-Bernstein-Z} and \ref{cor:Baryshnikov-Eisenberg-Stadje} will be explained later in Subsection \ref{subsec:Kac-Bernstein-probability-Z}.

\begin{table}[ht]
            \centering
            \caption{Correspondence table between Theorem \ref{thm:Z-characterize} and several related corollaries (``N'' denotes no corresponding condition)}
            \label{tab:Z-characterize-complex-T-correspond}

\vspace{8pt}
            
\begin{tabular}{c||c|c|c|c|c|c|c|c|c|c}
Theorem \ref{thm:Z-characterize} & \ref{item:Z-characterize-f1-zero} & \ref{item:Z-characterize-f2-zero} & \ref{item:Z-characterize-f1-one-point} & \ref{item:Z-characterize-f2-one-point} & \multicolumn{2}{c|}{\ref{item:Z-characterize-independent}} & \multicolumn{2}{c|}{\ref{item:Z-characterize-six-point}} & \multicolumn{2}{c}{\ref{item:Z-characterize-Gauss}} \\
\hline
Corollary \ref{cor:complex-T} & \ref{item:complex-T-f1-zero} & \ref{item:complex-T-f2-zero} & \ref{item:complex-T-f1-one-point} & \ref{item:complex-T-f2-one-point} & \ref{item:complex-T-independent-1-even} & \ref{item:complex-T-independent-1-odd} & \multicolumn{2}{c|}{\ref{item:complex-T-six-point}} & \ref{item:complex-T-Gauss} & \ref{item:complex-T-boundary} \\
\hline
Corollary \ref{cor:original-Kac-Bernstein-Z} & \multicolumn{2}{c|}{\multirow{2}{*}{N}} & \ref{item:original-Kac-Bernstein-Z-f1-one-point} & \ref{item:original-Kac-Bernstein-Z-f2-one-point} & \multicolumn{2}{c|}{\multirow{2}{*}{N}} & \ref{item:original-Kac-Bernstein-Z-six-point-f1-symmetry} & \ref{item:original-Kac-Bernstein-Z-six-point-f2-symmetry} & \multicolumn{2}{c}{\ref{item:original-Kac-Bernstein-Z-Gauss}} \\
\cline{1-1} \cline{4-5} \cline{8-11}
Corollary \ref{cor:Baryshnikov-Eisenberg-Stadje} & \multicolumn{2}{c|}{} & \ref{item:Baryshnikov-Eisenberg-Stadje-f1-one-point} & \ref{item:Baryshnikov-Eisenberg-Stadje-f2-one-point} & \multicolumn{2}{c|}{} & \multicolumn{2}{c|}{N} & \ref{item:Baryshnikov-Eisenberg-Stadje-Gauss} & \ref{item:Baryshnikov-Eisenberg-Stadje-Zk}
\end{tabular}
\end{table}

Similarly, the solution of Problem \ref{Q:Kac-Bernstein-complex} for $ X = \mathbb{ R } $ follows from Corollary \ref{cor:R-characterize} and Proposition \ref{prop:Kac-Bernstein-Fourier}:

\begin{corollary}
    \label{cor:Kac-Bernstein-complex-R}

A pair $ ( \mu_1 , \mu_2 ) $ of complex Borel measures on $ X = \mathbb{ R } $ satisfies the condition of Problem \ref{Q:Kac-Bernstein-complex} if and only if at least one of the following conditions holds.

\begin{enumerate}
  \item \label{item:Kac-Bernstein-complex-R-f1-zero}
  
  One has $ \mu_1 = 0 $.

  \item \label{item:Kac-Bernstein-complex-R-f2-zero}
  
  One has $ \mu_2 = 0 $.

\item \label{item:Kac-Bernstein-complex-R-Gauss}

There exist $ c_1 , c_2 , c_3 , c_4 , c_5 \in \mathbb{ C } $ with $ \re c_1 < 0 $ such that
\begin{align}
  \mu_1 & = e^{ c_1 x^2 + c_2 x + c_3 } dx, &
  \mu_2 & = e^{ c_1 x^2 + c_4 x + c_5 } dx
\end{align}
hold, where $ dx $ is the Lebesgue measure on $ \mathbb{ R } $.

\item \label{item:Kac-Bernstein-complex-R-degenerate}

There exist $ c_1 , c_2 \in \mathbb{ C } $ and $ t_1 , t_2 \in \mathbb{ R } $ such that
\begin{align}
  \mu_1 & = c_1 \delta_{ t_1 }, &
  \mu_2 & = c_2 \delta_{ t_2 }.
\end{align}

\end{enumerate}
    
\end{corollary}

\section{Kac--Bernstein theorem for probability measures}
\label{sec:Kac-Bernstein-probability}

In the case of probability measures, the main theorem and several corollaries in Section \ref{sec:main-relate} are well studied.
In this section, we consider the restriction of Problem \ref{Q:Kac-Bernstein-complex} to probability measures and review known results on the solutions of the original Kac--Bernstein functional equation \eqref{eq:Kac-Bernstein-original}.
That is, we consider the following problem:

\begin{problem}
  \label{Q:Kac-Bernstein-probability-represent}

Classify pairs $ ( \eta_1 , \eta_2 ) $ of independent random variables on a locally compact abelian group $ X $ such that the sum $ \eta_1 + \eta_2 $ and the difference $ \eta_1 - \eta_2 $ are also independent.
That is, classify pairs $ ( \mu_1 , \mu_2 ) $, where $ \mu_1 $ and $ \mu_2 $ are the probability distributions (probability Borel measures) of $ \eta_1 $ and $ \eta_2 $, respectively, such that the condition of Problem \ref{Q:Kac-Bernstein-complex} holds.

\end{problem}

There are many works concerning Problem \ref{Q:Kac-Bernstein-probability-represent} and its generalization, known as the Darmois--Skitovich theorem \cite{MR61322} \cite{skitovitch1953property}, in the literature \cite{MR137201} \cite{MR144373} \cite{MR294997} \cite{MR1065592} \cite{MR1210054} \cite{MR1618734} \cite{MR1785533} \cite{MR1981634} \cite{MR2144871} \cite{MR2354566} \cite{MR2481995} \cite{MR3136462} \cite{MR2790886} \cite{MR3109683} \cite{MR3028778} \cite{MR3145890} \cite{MR3626791} \cite{MR3873044} (see the books \cite{MR1206473} \cite{MR2412878} \cite{MR4566375} by Feldman for details).
Table \ref{tab:Kac-Bernstein-reference} summarizes the literature describing the solutions of Problem \ref{Q:Kac-Bernstein-probability-represent}.

\begin{table}[ht]
            \centering
            \caption{Literature describing the solutions of Problem \ref{Q:Kac-Bernstein-probability-represent}}
            \label{tab:Kac-Bernstein-reference}

\vspace{8pt}
            
\begin{tabular}{c|c}
locally compact abelian group $ X $ & literature \\
\hline
$ \mathbb{ R }^n $ & Kac, Bernstein \cite{MR371}, \cite{bernstein1941property} \\
When $ X_0 \cap \ker \alpha_X = \{ 0 \} $ & Feldman (Fact \ref{fact:Feldman-no-order-2}) \\
$ \mathbb{ T } $ & Baryshnikov--Eisenberg--Stadje (Corollary \ref{cor:Baryshnikov-Eisenberg-Stadje}) \\
$ \mathbb{ R } \times \mathbb{ T } $, $ \mathbf{ a } $-adic solenoid & Myronyuk \cite[Theorems 1 and 2]{MR2214813}
\end{tabular}
\end{table}

We introduce some terminology and notation to describe Feldman's results.

\begin{definition}

Let $ X $ be a locally compact abelian group.

\begin{enumerate}

    \item 
    
A probability Borel measure $ \mu $ on $ X $ is called a symmetric Gaussian measure if there exists a map $ \psi \colon \hat{ X } \to \mathbb{ C } $ such that
\begin{align}
    \hat{ \mu } ( u ) = e^{ \psi ( u ) }, & & \psi ( u + v ) + \psi ( u - v ) = 2 ( \psi ( u ) + \psi ( v ) )
\end{align}
hold for any $ u , v \in \hat{ X } $.
The set of all symmetric Gaussian measures on $ X $ is written as $ \Gamma^s ( X ) $.

    \item 

    The map $ \alpha_X \colon X \to X $ is defined as $ \alpha_X ( x ) \coloneqq x + x $ for $ x \in X $.
    The locally compact abelian group $ X $ is called a Corwin group if $ \alpha_X ( X ) = X $.

\item

The connected component of $ X $ containing the identity element $ 0 \in X $ is written as $ X_0 $.

\end{enumerate}
    
\end{definition}

When $ X_0 \cap \ker \alpha_X = \{ 0 \} $, Feldman obtained the solution of Problem \ref{Q:Kac-Bernstein-probability-represent} as follows:

\begin{fact}[{\cite[Theorem 4]{feld1987bernstein}, \cite[Theorems 9.19 and 9.21 and Proposition 9.20]{MR1206473}}]
    \label{fact:Feldman-no-order-2}

The following conditions on a locally compact abelian group $ X $ are all equivalent:

\begin{enumerate}
    \item \label{item:Feldman-no-order-2-general}

Two probability Borel measures $ \mu_1 $ and $ \mu_2 $ on $ X $ satisfy the condition of Problem \ref{Q:Kac-Bernstein-probability-represent} if and only if there exist $ x_1 , x_2 \in X $, a compact Corwin subgroup $ K \subset X $, and $ \gamma \in \Gamma^s ( X ) $ such that
    \begin{align}
        \mu_1 = \delta_{ x_1 } * dm_K * \gamma, & &
        \mu_2 = \delta_{ x_2 } * dm_K * \gamma.
    \end{align}

\item \label{item:Feldman-no-order-2-T}

For any compact Corwin group $ K \subset X $, its quotient group $ X / K $ has no subgroup which is isomorphic to $ \mathbb{ T } $ as a topological group.

\item \label{item:Feldman-no-order-2-Corwin}

For any compact Corwin group $ K \subset X $, its Pontryagin dual $ \hat{ K } $ is also a Corwin group.

\item \label{item:Feldman-no-order-2-order}

One has $ X_0 \cap \ker \alpha_X = \{ 0 \} $.

\end{enumerate}
    
\end{fact}

Rukhin proved that \ref{item:Feldman-no-order-2-general} holds in the case where $ \alpha_X \colon X \to X $ is an isomorphism \cite{rukhin1969theorem}.
If $ \alpha_X $ is an isomorphism, then \ref{item:Feldman-no-order-2-order} holds.
Thus, Rukhin's result follows from Fact \ref{fact:Feldman-no-order-2} $ \text{ \ref{item:Feldman-no-order-2-order} } \Longrightarrow \text{ \ref{item:Feldman-no-order-2-general} } $.
The original Kac--Bernstein functional equation appears in the proof of Fact \ref{fact:Feldman-no-order-2} by Rukhin and Feldman.
That is, similarly to Proposition \ref{prop:Kac-Bernstein-Fourier}, Problem \ref{Q:Kac-Bernstein-probability-represent} can be reformulated by using \eqref{eq:Kac-Bernstein-original} on $ \hat{ X } $ as follows (see also \cite[Lemma 4.2]{MR4566375}):

\begin{corollary}[{\cite{MR428375}}]
    \label{cor:original-Kac-Bernstein-Fourier}

The following conditions for a pair $ ( \mu_1 , \mu_2 ) $ of probability Borel measures on a locally compact abelian group $ X $ are equivalent:

\begin{enumerate}
  \item \label{item:original-Kac-Bernstein-Fourier-independent}

  The pair $ ( \mu_1 , \mu_2 ) $ satisfies the condition of Problem \ref{Q:Kac-Bernstein-probability-represent}, that is, $ F_* ( \mu_1 \times \mu_2 ) $ is the product measure of two probability measures.
  
  \item \label{item:original-Kac-Bernstein-Fourier-equation}

One has \eqref{eq:Kac-Bernstein-original} for any $ u , v \in \hat{ X } $, where $ f_1 , f_2 \colon \hat{ X } \to \mathbb{ C } $ are defined by \eqref{eq:f1-f2-Fourier}.

\end{enumerate}
    
\end{corollary}

The proof of Fact \ref{fact:Feldman-no-order-2} by Rukhin and Feldman essentially uses Corollary \ref{cor:original-Kac-Bernstein-Fourier} $ \text{ \ref{item:original-Kac-Bernstein-Fourier-independent} } \Longrightarrow \text{ \ref{item:original-Kac-Bernstein-Fourier-equation} }$.
Feldman gave a necessary condition for functions $ f_1 , f_2 \colon G \to \mathbb{ C } $ (which may not be the Fourier--Stieltjes transforms of complex measures) that satisfy \eqref{eq:Kac-Bernstein-original} under the assumptions that $ \supp f_1 = \supp f_2 = G $ and
\begin{align}
    f_1 ( - u ) & = \overline{ f_1 ( u ) }, &
    f_2 ( - u ) = \overline{ f_2 ( u ) } \label{eq:Hermitian-nonvanish}
\end{align}
for any $ u \in G $ \cite[Theorem 3.1]{MR4285097}.
The Fourier--Stieltjes transforms $ f_1 $ and $ f_2 $ of any probability Borel measures $ \mu_1 $ and $ \mu_2 $ satisfy \eqref{eq:probability-dual} and hence Corollary \ref{cor:original-Kac-Bernstein-Fourier} follows from Proposition \ref{prop:Kac-Bernstein-Fourier} and the following proposition:

\begin{proposition}
    \label{prop:original-Kac-Bernstein}

The following condition \ref{item:original-Kac-Bernstein-original} implies \ref{item:original-Kac-Bernstein-general} for any functions $ f_1 , f_2 \colon G \to \mathbb{ C } $ on a locally compact abelian group $ G $.

\begin{enumerate}

\item \label{item:original-Kac-Bernstein-original}

One has \eqref{eq:Kac-Bernstein-original} for any $ u , v \in G $.
\item \label{item:original-Kac-Bernstein-general}

One has \eqref{eq:complex-Kac-Bernstein} for any $ u , u' , v , v' \in G $.
\end{enumerate}

Furthermore, if \eqref{eq:probability-dual} holds, then \ref{item:original-Kac-Bernstein-original} and \ref{item:original-Kac-Bernstein-general} are equivalent.

\end{proposition}

\begin{proof}

Since 
\begin{align}
    & \phantom{ {} = {} } f_1 ( u + v ) f_2 ( u - v ) f_1 ( u' + v' ) f_2 ( u' - v' ) \\
    & = f_1 ( u ) f_2 ( u ) f_1 ( v ) f_2 ( - v ) f_1 ( u' ) f_2 ( u' ) f_1 ( v' ) f_2 ( - v' ) \\
    & = f_1 ( u + v' ) f_2 ( u - v' ) f_1 ( u' + v ) f_2 ( u' - v )
\end{align}
holds by \eqref{eq:Kac-Bernstein-original}, we have \ref{item:original-Kac-Bernstein-general} by \ref{item:original-Kac-Bernstein-original}.

By substituting $ u' = v' = 0 $ into \eqref{eq:complex-Kac-Bernstein}, we have
\begin{align}
    f_1 ( u + v ) f_2 ( u - v ) f_1 ( 0 ) f_2 ( 0 )
    = f_1 ( u ) f_2 ( u ) f_1 ( v ) f_2 ( - v ).
\end{align}
Thus, if \eqref{eq:probability-dual} holds, then \ref{item:original-Kac-Bernstein-original} follows from \ref{item:original-Kac-Bernstein-general}.
\end{proof}

The Fourier--Stieltjes transforms $ f_1 $ and $ f_2 $ of probability Borel measures $ \mu_1 $ and $ \mu_2 $ satisfy \eqref{eq:probability-dual} and \eqref{eq:Hermitian-nonvanish}.
Thus, if $ f_1 $ and $ f_2 $ satisfy the equivalent conditions of Proposition \ref{prop:original-Kac-Bernstein}, then the intersection $ \supp f_1 \cap \supp f_2 $ is an open subgroup of $ G $ (see \cite[Theorem 1]{MR3626791}).
On the other hand, $ \supp f_1 \cap \supp f_2 $ may not be an open subgroup of $ G $ when $ \mu_1 $ and $ \mu_2 $ are complex Borel measures.
This point makes Problem \ref{Q:Kac-Bernstein-complex} more difficult than Problem \ref{Q:Kac-Bernstein-probability-represent}.

In Sections \ref{subsec:Kac-Bernstein-probability-Z} and \ref{subsec:Kac-Bernstein-probability-R}, we see the Kac--Bernstein theorem in more detail when $ G = \mathbb{ Z } $ and $ G = \mathbb{ R } $, respectively.

\subsection{Case of \texorpdfstring{$ G = \mathbb{ Z } $}{G = Z}}
\label{subsec:Kac-Bernstein-probability-Z}

When $ G = \mathbb{ Z } $, the solutions of the complex Kac--Bernstein functional equation \eqref{eq:complex-Kac-Bernstein} are given in Theorem \ref{thm:Z-characterize}.
Thus, the functions satisfying the equivalent conditions of Proposition \ref{prop:original-Kac-Bernstein} can be classified.
That is, we obtain the following corollary which is a specialization of the classification in Theorem \ref{thm:Z-characterize} to the case where \eqref{eq:probability-dual} holds:

\begin{corollary}
    \label{cor:original-Kac-Bernstein-Z}

When $ G = \mathbb{ Z } $, two functions $ f_1 , f_2 \colon \mathbb{ Z } \to \mathbb{ C } $ with \eqref{eq:probability-dual} satisfy the equivalent conditions of Proposition \ref{prop:original-Kac-Bernstein} if and only if at least one of the following conditions holds:

\begin{enumerate}

\item \label{item:original-Kac-Bernstein-Z-f1-one-point}

One has $ \supp f_1 = \supp f_2 \cap 2 \mathbb{ Z } = \{ 0 \} $.

\item \label{item:original-Kac-Bernstein-Z-f2-one-point}

One has $ \supp f_1 \cap 2 \mathbb{ Z } = \supp f_2 = \{ 0 \} $.

\item \label{item:original-Kac-Bernstein-Z-six-point-f1-symmetry}

There exists $ k  \in 1 + 2 \mathbb{ Z } $ such that
\begin{align}
    \supp f_1 \subset \{ 0 , \pm k \}, & &
    \supp f_2 \subset \{ 0 , k , 2 k \}, & &
    f_2 ( 2 k ) = f_2 ( k )^2 f_1 ( - k ) f_1 ( k ).
\end{align}

\item \label{item:original-Kac-Bernstein-Z-six-point-f2-symmetry}

There exists $ k  \in 1 + 2 \mathbb{ Z } $ such that
\begin{align}
    \supp f_1 \subset \{ 0 , k , 2 k \}, & &
    \supp f_2 \subset \{ 0 , \pm k \}, & &
    f_1 ( 2 k ) = f_1 ( k )^2 f_2 ( - k ) f_2 ( k ).
\end{align}

\item \label{item:original-Kac-Bernstein-Z-Gauss}

There exist $c_1 , c_2 , c_3 , c_5 \in \mathbb{ C } $ and $ k = 1 , 3 , 5 , \cdots $ such that
\begin{align}
  f_1 ( u )
  & =
  \left\{
  \begin{aligned}
    & e^{ c_1 u^2 + c_2 ( ( - 1 )^u - 1 ) + c_3 u } & & \text{ if $ u \in k \mathbb{ Z } $} \\
    & 0 & & \text{ if $ u \notin k \mathbb{ Z } $ }
  \end{aligned}
  \right. , \\
  f_2 ( u )
  & =
  \left\{
  \begin{aligned}
    & e^{ c_1 u^2 - c_2 ( ( - 1 )^u - 1 ) + c_5 u } & & \text{ if $ u \in k \mathbb{ Z } $} \\
    & 0 & & \text{ if $ u \notin k \mathbb{ Z } $ }
  \end{aligned}
  \right. . \label{eq:original-Kac-Bernstein-Z-Gauss-represent}
\end{align}

\end{enumerate}
    
\end{corollary}

When $ X = \mathbb{ T } $, the Fourier--Stieltjes transforms of two probability Borel measures which satisfy the condition of Problem \ref{Q:Kac-Bernstein-probability-represent} can be classified by using Corollary \ref{cor:original-Kac-Bernstein-Z}.
That is, an alternative proof of the following corollary proved by  Baryshnikov--Eisenberg--Stadje can be obtained.
This classification can be regarded as the specialization of Corollary \ref{cor:complex-T} from complex measures to probability measures.

\begin{corollary}[{\cite[Theorem 2]{MR1221915}}]
  \label{cor:Baryshnikov-Eisenberg-Stadje}

A pair $ ( \mu_1 , \mu_2 ) $ of probability Borel measures on $ X = \mathbb{ T } $ satisfies the condition of Problem \ref{Q:Kac-Bernstein-probability-represent} if and only if at least one of the following conditions holds:

\begin{enumerate}
\item \label{item:Baryshnikov-Eisenberg-Stadje-f1-one-point}

One has $ \mu_1 = dm_{ \mathbb{ T } } $ and $ \mu_2 - dm_{ \mathbb{ T } } \in M_- $.

\item \label{item:Baryshnikov-Eisenberg-Stadje-f2-one-point}

One has $ \mu_2 = dm_{ \mathbb{ T } } $ and $ \mu_1 - dm_{ \mathbb{ T } } \in M_- $.

\item \label{item:Baryshnikov-Eisenberg-Stadje-Gauss}

There exist $ c_2 \in \mathbb{ R } $, $ r > 0 $, $ \theta_1 , \theta_2 \in \mathbb{ T } $, and $ k = 1 , 3 , 5 , \cdots $ such that
\begin{align}
\mu_1 & = e^{ - c_2 } \left( \cosh ( c_2 ) \vartheta \left( \frac{ k x + \theta_1 }{ 2 \pi } , i r \right) + \sinh ( c_2 ) \vartheta \left( \frac{ k x + \theta_1 + \pi }{ 2 \pi } , i r \right) \right) dm_{ \mathbb{ T } } ( x ), \\
\mu_2 & = e^{ c_2 } \left( \cosh ( c_2 ) \vartheta \left( \frac{ k x + \theta_2 }{ 2 \pi } , i r \right) - \sinh ( c_2 ) \vartheta \left( \frac{ k x + \theta_2 + \pi }{ 2 \pi } , i r \right) \right) dm_{ \mathbb{ T } } ( x ).
\end{align}

\item \label{item:Baryshnikov-Eisenberg-Stadje-Zk}

There exist $ \theta_1 , \theta_2 \in \mathbb{ T } $ and $ k = 1 , 3 , 5 , \cdots $ such that
\begin{align}
    \mu_1 = \frac{ 1 }{ k } \sum_{ l = 0 }^{ k - 1 } \delta_{ \theta_1 + 2 \pi l / k }, & &
    \mu_2 = \frac{ 1 }{ k } \sum_{ l = 0 }^{ k - 1 } \delta_{ \theta_2 + 2 \pi l / k }.
\end{align}

\end{enumerate}
    
\end{corollary}

The correspondence between the conditions of Corollaries \ref{cor:original-Kac-Bernstein-Z} and \ref{cor:Baryshnikov-Eisenberg-Stadje} and those of Theorem \ref{thm:Z-characterize} and Corollary \ref{cor:complex-T} is summarized in Table \ref{tab:Z-characterize-complex-T-correspond} in Section \ref{sec:main-relate}.

\subsection{Case of \texorpdfstring{$ G = \mathbb{ R } $}{G = R}}
\label{subsec:Kac-Bernstein-probability-R}

When $ G = \mathbb{ R } $, Corollary \ref{cor:R-characterize} gives the solutions of the complex Kac--Bernstein functional equation \eqref{eq:complex-Kac-Bernstein}.
Thus, the functions satisfying the equivalent conditions of Proposition \ref{prop:original-Kac-Bernstein} can be classified.
That is, we obtain the following corollary which is the specialization of the classification in Corollary \ref{cor:R-characterize} to the case where \eqref{eq:probability-dual} holds:

\begin{corollary}
    \label{cor:original-Kac-Bernstein-R}

When $ G = \mathbb{ R } $, continuous functions $ f_1 , f_2 \colon \mathbb{ R } \to \mathbb{ C } $ with \eqref{eq:probability-dual} satisfy the equivalent conditions of Proposition \ref{prop:original-Kac-Bernstein} if and only if there exist $ d_1 , d_2 , d_3 \in \mathbb{ C } $ such that
\begin{align}
    f_1 ( u ) & = e^{ d_1 u^2 + d_2 u }, &
    f_2 ( u ) = e^{ d_1 u^2 + d_3 u } 
\end{align}
hold for any $ u \in \mathbb{ R } $.
    
\end{corollary}

Corollary \ref{cor:original-Kac-Bernstein-R} classifies the functions satisfying \eqref{eq:complex-Kac-Bernstein}, which is more general than \eqref{eq:Kac-Bernstein-original}, under the continuity assumption.
On the other hand, Kannappan classified the positive functions $ f_1 , f_2 \colon \mathbb{ R } \to ( 0 , \infty ) $ (which may not be continuous) with \eqref{eq:Kac-Bernstein-original} \cite[Corollary 1]{MR1792081}.

The original Kac--Bernstein theorem can be obtained by restricting Corollary \ref{cor:Kac-Bernstein-complex-R} to probability Borel measures as follows:

\begin{corollary}[Kac--Bernstein theorem for $ \mathbb{ R } $, {\cite{MR371}, \cite{bernstein1941property}}]
    \label{cor:original-Kac-Bernstein}

A pair $ ( \mu_1 , \mu_2 ) $ of probability Borel measures on $ X = \mathbb{ R } $ satisfies the condition of Problem \ref{Q:Kac-Bernstein-probability-represent} if and only if either of the following conditions holds:

\begin{enumerate}

\item \label{item:original-Kac-Bernstein-R-Gauss}

There exist $ A > 0 $ and $ B_1 , B_2 \in \mathbb{ R } $ such that
\begin{align}
  \mu_1 & = \sqrt{ \frac{ A }{ \pi } } e^{ - A ( x + B_1 )^2 } dx, &
  \mu_2 & = \sqrt{ \frac{ A }{ \pi } } e^{ - A ( x + B_2 )^2 } dx.
\end{align}

\item \label{item:original-Kac-Bernstein-degenerate}

There exist $ t_1 , t_2 \in \mathbb{ R } $ such that
\begin{align}
  \mu_1 & = \delta_{ t_1 }, &
  \mu_2 & = \delta_{ t_2 }.
\end{align}

\end{enumerate}

\end{corollary}

Even when $ X = \mathbb{ R }^n $, Kac and Bernstein proved independently that the random variables satisfying the condition of Problem \ref{Q:Kac-Bernstein-probability-represent} are normally distributed.

\section{Proof of Theorem \ref{thm:Z-characterize}}
\label{sec:Z-characterize-proof}

In this section, we prove Theorem \ref{thm:Z-characterize}.
We prepare several lemmas of Theorem \ref{thm:Z-characterize} in Subsection \ref{subsec:Z-characterize-proof-lemma}, and give a proof of Theorem \ref{thm:Z-characterize} in Subsection \ref{subsec:Z-characterize-proof-complete}.

\subsection{Key lemmas of Theorem \ref{thm:Z-characterize}}
\label{subsec:Z-characterize-proof-lemma}

In this subsection, we show several lemmas to prove Theorem \ref{thm:Z-characterize}.
The key idea is to analyze the supports of $ f_1 $ and $ f_2 $ using the six-point identity (Lemma \ref{lem:six-point-equation}), and then to apply the uniqueness result (Lemma \ref{lem:supp-all}) on suitable arithmetic progressions.
First, we prove that \eqref{eq:Z-characterize-six-point-equation} is a special case of \eqref{eq:complex-Kac-Bernstein}:

\begin{lemma}
    \label{lem:six-point-equation}

If functions $ f_1 , f_2 \colon \mathbb{ Z } \to \mathbb{ C } $ satisfy \eqref{eq:complex-Kac-Bernstein} for any $ u , u' , v , v' \in \mathbb{ Z } $, then \eqref{eq:Z-characterize-six-point-equation} holds for any $ a_1 , a_2 , k \in \mathbb{ Z } $ with
\begin{align}
   a_1 + a_2 + k \in 2 \mathbb{ Z }. \label{eq:six-point-equation-assume}
\end{align}

\end{lemma}

\begin{proof}

Put
\begin{align}
    u \coloneqq \frac{ a_1 + a_2 + k }{ 2 }, & &
    v \coloneqq \frac{ a_1 - a_2 + k }{ 2 }, & &
    u' \coloneqq \frac{ a_1 + a_2 - k }{ 2 }, & &
    v' \coloneqq \frac{ a_1 - a_2 - k }{ 2 }.
\end{align}
We have $ u \in \mathbb{ Z } $ by \eqref{eq:six-point-equation-assume}.
Moreover, we have
\begin{align}
    v &= u - a_2, &
    u' &= u - k, &
    v' &= v - k.
\end{align}
Thus, these values $ u , v , u' , v' $ are all integers.
Substituting these integers into
\eqref{eq:complex-Kac-Bernstein}
yields \eqref{eq:Z-characterize-six-point-equation}.
\end{proof}

By Lemma \ref{lem:six-point-equation}, the supports of functions with \eqref{eq:complex-Kac-Bernstein} satisfy the following conditions:

\begin{lemma}
   \label{lem:supp-arithmetic-progression}

Let $ f_1 $ and $ f_2 $ be as in Lemma \ref{lem:six-point-equation}, and
\begin{align}
    S_1 ( a , j ) \coloneqq \supp f_1 \cap ( a + j \mathbb{ Z } ), & &
    S_2 ( a , j ) \coloneqq \supp f_2 \cap ( a + j \mathbb{ Z } ) \label{eq:supp-arithmetic-progression-S-define}
\end{align}
for $ a , j \in \mathbb{ Z } $.
Suppose $ a_2 , k \in \mathbb{ Z } $ satisfy $ a_2 \pm k \in \supp f_2 $ and let $ a_1 \in S_1 ( a_2 + k , 2 ) $.

\begin{enumerate}
    \item \label{item:supp-arithmetic-progression-supp}

    One has
    \begin{align}
        a_1 \pm k \in S_1 ( a_2 , 2 ), & &
        a_2 \in S_2 ( a_1 + k , 2 ). \label{eq:supp-arithmetic-progression-supp-include}
    \end{align}

\item \label{item:supp-arithmetic-progression-main}

Furthermore, if $ k \in 2 \mathbb{ Z } $ holds, then one has
\begin{align}
    a_1 + ( k / 2 ) \mathbb{ Z } = S_1 ( a_1 , k / 2 ), & &
    a_2 + ( k / 2 ) \mathbb{ Z } = S_2 ( a_2 , k / 2 ).
\end{align}

\end{enumerate}

\end{lemma}

\begin{proof}

\begin{enumerate}
    \item 
    Since \eqref{eq:six-point-equation-assume} holds by $ a_1 \in S_1 ( a_2 + k , 2 ) $, we have \eqref{eq:Z-characterize-six-point-equation} by Lemma \ref{lem:six-point-equation}.
    Thus, it follows from
    \begin{align}
        a_1 \in S_1 ( a_2 + k , 2 ) \subset \supp f_1 , & & a_2 \pm k \in \supp f_2
    \end{align}
    that
    \begin{align}
        0
        \neq f_1 ( a_1 )^2 f_2 ( a_2 - k ) f_2 ( a_2 + k )
        = f_2 ( a_2 )^2 f_1 ( a_1 - k ) f_1 ( a_1 + k )
    \end{align}
    and hence we obtain \eqref{eq:supp-arithmetic-progression-supp-include} by \eqref{eq:six-point-equation-assume}.

    \item
It suffices to show
\begin{align}
    a_1 + ( k / 2 ) \mathbb{ Z } \subset \supp f_1, & &
    a_2 + ( k / 2 ) \mathbb{ Z } \subset \supp f_2
\end{align}
by \eqref{eq:supp-arithmetic-progression-S-define}.

First, we prove $ a_1 \pm k L \in \supp f_1 $ for any $ L = 0 , 1 , 2 , \cdots $ by induction.
When $ L = 0 $, we have
\begin{align}
     a_1 \pm k L = a_1 \in S_1 ( a_2 + k , 2 ) \subset \supp f_1.
\end{align}
Here, for any $ L = 0 , 1 , 2 , \cdots $, it suffices to show $ a_1 \pm k ( L + 1 ) \in \supp f_1 $ if $ a_1 \pm k L \in \supp f_1 $ holds.
Since
\begin{align}
    a_1 \pm k L \in a_1 + 2 \mathbb{ Z } = a_2 + k + 2 \mathbb{ Z }
\end{align}
holds by $ k \in 2 \mathbb{ Z } $ and \eqref{eq:six-point-equation-assume}, we have $ a_1 \pm k L \in S_1 ( a_2 + k , 2 ) $ by $ a_1 \pm k L \in \supp f_1 $.
Thus, we have
\begin{align}
    a_1 \pm k ( L + 1 ) \in S_1 ( a_2 , 2 ) \subset \supp f_1
\end{align}
by replacing $ a_1 $ with $ a_1 \pm k L $ and applying \ref{item:supp-arithmetic-progression-supp}.
Therefore, we obtain $ a_1 + k \mathbb{ Z } \subset \supp f_1 $.

Next, we prove $ a_1 + ( k / 2 ) \mathbb{ Z } \subset \supp f_1 $.
It suffices to show $ a_1 + ( L' + 1 / 2 ) k \in \supp f_1 $ for any $ L' \in \mathbb{ Z } $ by $ a_1 \pm k \mathbb{ Z } \subset \supp f_1 $.
We have $ a_1 + k L' , a_1 + k ( L' + 1 ) \in \supp f_1 $ by $ a_1 \pm k \mathbb{ Z } \subset \supp f_1 $, and 
\begin{align}
    a_2 \in S_2 ( a_1 + k , 2 ) = S_2 ( a_1 + ( L' + 1 ) k , 2 )
\end{align}
by $ k \in 2 \mathbb{ Z } $.
Thus, it follows from applying \ref{item:supp-arithmetic-progression-supp} with $ (f_1 , f_2 , a_1 , a_2 , k) $ replaced by
$ (f_2 , f_1 , a_2 , a_1 + k ( L' + 1 / 2 ) , k / 2 ) $ that
\begin{align}
    a_1 + ( L' + 1 / 2 ) k \in S_1 ( a_2 + k / 2 , 2 ) \subset \supp f_1.
\end{align}
Thus, we obtain $ a_1 + ( k / 2 ) \mathbb{ Z } \subset \supp f_1 $.

In addition, we have
\begin{align}
    a_1 \pm k \in \supp f_1, & & a_2 \in S_2 ( a_1 + k , 2 )
\end{align}
and hence $ a_2 + ( k / 2 ) \mathbb{ Z } \subset \supp f_2 $ by replacing $ f_1 $ and $ a_1 $ with $ f_2 $ and $ a_2 $, respectively.
\qedhere

\end{enumerate}
    
\end{proof}

Next, we confirm that the condition \ref{item:Z-characterize-Gauss} of Theorem \ref{thm:Z-characterize} satisfies \eqref{eq:complex-Kac-Bernstein} when $ k = 1 $:

\begin{lemma}
    \label{lem:Gauss-Kac-Bernstein-observe}

For any $ c_1 , c_2 , c_3 , c_4 , c_5 , c_6 \in \mathbb{ C } $, the functions $ f_1 , f_2 \colon \mathbb{ Z } \to \mathbb{ C } $ defined as
\begin{align}
  f_1 ( u )
  = e^{ c_1 u^2 + c_2 ( - 1 )^u + c_3 u + c_4 }, & &
  f_2 ( u )
  = e^{ c_1 u^2 - c_2 ( - 1 )^u + c_5 u + c_6 } \label{eq:supp-all-Gauss-represent}
\end{align}
satisfy \eqref{eq:complex-Kac-Bernstein} for any $ u , u' , v , v' \in \mathbb{ Z } $.

\end{lemma}

\begin{proof}
The functions $ g_1 , g_2 \colon \mathbb{ Z } \to \mathbb{ C } $ are defined as
\begin{align}
    g_1 ( u ) \coloneqq c_1 u^2 + c_2 ( - 1 )^u + c_3 u + c_4, & &
    g_2 ( u ) \coloneqq c_1 u^2 - c_2 ( - 1 )^u + c_5 u + c_6.
\end{align}
Then it suffices to show
\begin{align}
& \phantom{ { } = { } } g_1 ( u + v ) + g_2 ( u - v ) + g_1 ( u' + v' ) + g_2 ( u' - v' ) \\
& = g_1 ( u + v' ) + g_2 ( u - v' ) + g_1 ( u' + v ) + g_2 ( u' - v ).
\label{eq:Gauss-Kac-Bernstein-observe-proof-g}
\end{align}
By decomposing into components, it suffices to show the case when
\begin{align}
    ( g_1 ( u ) , g_2 ( u ) ) = ( u^2 , u^2 ) , \; ( ( - 1 )^{ u } , ( - 1 )^{ u + 1 } ) , \; ( c_3 u + c_4 , c_5 u + c_6 ).
\end{align}

\begin{enumerate}[label=(Case \arabic*)]
    \item
    
    When $ ( g_1 ( u ) , g_2 ( u ) ) = ( u^2 , u^2 ) $:

    We obtain
\begin{align}
& \phantom{ { } = { } } g_1 ( u + v ) + g_2 ( u - v ) + g_1 ( u' + v' ) + g_2 ( u' - v' ) \\
& = ( u + v )^2 + ( u - v )^2 + ( u' + v' )^2 + ( u' - v' )^2 \\
& = 2 ( u^2 + v^2 + {u'}^2 + {v'}^2 ) \\
& = g_1 ( u + v' ) + g_2 ( u - v' ) + g_1 ( u' + v ) + g_2 ( u' - v ).
\end{align}

\item

When $ ( g_1 ( u ) , g_2 ( u ) ) = ( ( - 1 )^{ u } , ( - 1 )^{ u + 1 } ) $:

We have
\begin{align}
    g_1 ( u + v ) + g_2 ( u - v )
    = ( - 1 )^{ u + v } + ( - 1 )^{ u - v + 1 }
    = 0
\end{align}
by
\begin{align}
    ( u + v ) + ( u - v + 1 ) = 2 u + 1 \in 1 + 2 \mathbb{ Z }.
\end{align}
Similarly, we have
\begin{align}
    g_1 ( u' + v' ) + g_2 ( u' - v' )
    = g_1 ( u + v' ) + g_2 ( u - v' )
    = g_1 ( u' + v ) + g_2 ( u' - v )
    = 0
\end{align}
and hence
\begin{align}
    & \phantom{ { } = { } } g_1 ( u + v ) + g_2 ( u - v ) + g_1 ( u' + v' ) + g_2 ( u' - v' ) \\
    & = 0 \\
    & = g_1 ( u + v' ) + g_2 ( u - v' ) + g_1 ( u' + v ) + g_2 ( u' - v ).
\end{align}

\item
When $ ( g_1 ( u ) , g_2 ( u ) ) = ( c_3 u + c_4 , c_5 u + c_6 ) $:

Since
\begin{align}
g_1 ( u + v ) + g_1 ( u' + v' )
& = c_3 ( u + v + u' + v' ) + 2 c_4
= g_1 ( u + v' )  + g_1 ( u' + v ), \\
g_2 ( u - v ) + g_2 ( u' - v' )
& = c_5 ( u - v + u' - v' ) + 2 c_6
= g_2 ( u - v' )  + g_2 ( u' - v )
\end{align}
hold, we have \eqref{eq:Gauss-Kac-Bernstein-observe-proof-g}.

\end{enumerate}

Thus, we obtain \eqref{eq:complex-Kac-Bernstein}.
\end{proof}

By using Lemmas \ref{lem:six-point-equation} and \ref{lem:Gauss-Kac-Bernstein-observe}, the values of two functions satisfying the complex Kac--Bernstein functional equation \eqref{eq:complex-Kac-Bernstein} are identified as follows:

\begin{lemma}
  \label{lem:supp-all}

For any $ q_1 , q_2 , q_3 , q_4 , q_5 , q_6 \in \mathbb{ C } $, the following conditions on $ f_1 , f_2 \colon \mathbb{ Z } \to \mathbb{ C } $ are equivalent.

\begin{enumerate}
    \item \label{item:supp-all-Kac-Bernstein}

    One has \eqref{eq:complex-Kac-Bernstein} for any $ u , u' , v , v' \in \mathbb{ Z } $ and
    \begin{align}
    f_1 ( - 1 ) = e^{ q_1 } , & &
    f_1 ( 0 ) = e^{ q_2 } , & &
    f_1 ( 1 ) = e^{ q_3 } , & &
    f_2 ( - 1 ) = e^{ q_4 } , & &
    f_2 ( 0 ) = e^{ q_5 } , & &
    f_2 ( 1 ) = e^{ q_6 }. \label{eq:supp-all-Kac-Bernstein-exponent}
\end{align}

\item \label{item:supp-all-Gauss}

One has \eqref{eq:supp-all-Gauss-represent} for any $ u \in \mathbb{ Z } $ when $ c_1 , c_2 , c_3 , c_4 , c_5 , c_6 \in \mathbb{ C } $ are defined as
\begin{align}
    P \coloneqq \frac{ 1 }{ 8 }
    \begin{pmatrix}
    2 & - 4 & 2 & 2 & - 4 & 2 \\
    - 1 & 2 & - 1 & 1 & - 2 & 1 \\
    - 4 & 0 & 4 & 0 & 0 & 0 \\
    1 & 6 & 1 & - 1 & 2 & - 1 \\
    0 & 0 & 0 & - 4 & 0 & 4 \\
    - 1 & 2 & - 1 & 1 & 6 & 1
  \end{pmatrix}, & &
  \begin{pmatrix}
    c_1 \\
    c_2 \\
    c_3 \\
    c_4 \\
    c_5 \\
    c_6
  \end{pmatrix}
  \coloneqq P
  \begin{pmatrix}
    q_1 \\
    q_2 \\
    q_3 \\
    q_4 \\
    q_5 \\
    q_6
  \end{pmatrix}
  .
  \label{eq:supp-all-Gauss-const-def}
\end{align}
    
\end{enumerate}

\end{lemma}

\begin{proof}

We prove $ \text{ \ref{item:supp-all-Kac-Bernstein} } \Longrightarrow \text{ \ref{item:supp-all-Gauss} } $ by induction on $ | u | $.

First, we show \eqref{eq:supp-all-Gauss-represent} when $ | u | = 0 , 1 $.
The matrix $ P $ is invertible, and we have
\begin{align}
    P^{ - 1 } =
    \begin{pmatrix}
    1 & - 1 & - 1 & 1 & 0 & 0 \\
    0 & 1 & 0 & 1 & 0 & 0 \\
    1 & - 1 & 1 & 1 & 0 & 0 \\
    1 & 1 & 0 & 0 & - 1 & 1 \\
    0 & - 1 & 0 & 0 & 0 & 1 \\
    1 & 1 & 0 & 0 & 1 & 1
  \end{pmatrix}
  .
\end{align}
Thus, it follows from \eqref{eq:supp-all-Gauss-const-def} that
\begin{align}
    \begin{pmatrix}
    q_1 \\
    q_2 \\
    q_3 \\
    q_4 \\
    q_5 \\
    q_6
  \end{pmatrix}
  = P^{ - 1 }
  \begin{pmatrix}
    c_1 \\
    c_2 \\
    c_3 \\
    c_4 \\
    c_5 \\
    c_6
  \end{pmatrix}
  =
  \begin{pmatrix}
    c_1 - c_2 - c_3 + c_4 \\
    c_2 + c_4 \\
    c_1 - c_2 + c_3 + c_4 \\
    c_1 + c_2 - c_5 + c_6 \\
    - c_2 + c_6 \\
    c_1 + c_2 + c_5 + c_6
  \end{pmatrix}
  \label{eq:supp-all-proof-inverse}
\end{align}
and hence we obtain \eqref{eq:supp-all-Gauss-represent} by \eqref{eq:supp-all-Kac-Bernstein-exponent}.

Next, we assume that \eqref{eq:supp-all-Gauss-represent} holds for any $ | u | \leq L $, and we prove \eqref{eq:supp-all-Gauss-represent} for $ | u | = L + 1 $ (where $ L = 1 , 2 , \cdots $).
If necessary, by replacing $ f_1 ( u ) $ and $ f_2 ( u ) $ with $ f_1 ( - u ) $ and $ f_2 ( - u ) $, respectively, it suffices to show \eqref{eq:supp-all-Gauss-represent} when $ u = L + 1 $.
Then we have
\begin{align}
     f_1 ( L )^2 f_2 ( L - 2 ) f_2 ( L )
     & = f_2 ( L - 1 )^2 f_1 ( L - 1 ) f_1 ( L + 1 ) \label{eq:supp-all-proof-simple-equation}
\end{align}
by Lemma \ref{lem:six-point-equation}.
Since $ | L - 2 | , | L - 1 | , | L | \leq L $ hold by $ L \geq 1 $, the value of $ f_1 ( L + 1 ) $ satisfying \eqref{eq:supp-all-proof-simple-equation} is unique by the induction hypothesis.
Thus, we have $ f_1 ( u ) = e^{ c_1 u^2 + c_2 ( - 1 )^u + c_3 u + c_4 } $ for $ u = L + 1 $ by Lemma \ref{lem:Gauss-Kac-Bernstein-observe}.
Similarly, the equation $ f_2 ( u ) = e^{ c_1 u^2 - c_2 ( - 1 )^u + c_5 u + c_6 } $ holds and hence we have \eqref{eq:supp-all-Gauss-represent}.
Then we obtain \ref{item:supp-all-Gauss}.

Finally, we prove $ \text{ \ref{item:supp-all-Gauss} } \Longrightarrow \text{ \ref{item:supp-all-Kac-Bernstein} } $.
We have \eqref{eq:complex-Kac-Bernstein} for any $ u , v , u' , v' \in \mathbb{ Z } $ by Lemma \ref{lem:Gauss-Kac-Bernstein-observe}.
In addition, we have \eqref{eq:supp-all-proof-inverse} by \eqref{eq:supp-all-Gauss-const-def} and hence \eqref{eq:supp-all-Kac-Bernstein-exponent} holds.
Then we obtain \ref{item:supp-all-Kac-Bernstein}.
\end{proof}

The following lemma implies that Lemmas \ref{lem:supp-arithmetic-progression} and \ref{lem:supp-all} are key lemmas of Theorem \ref{thm:Z-characterize}:

\begin{lemma}
    \label{lem:scaling}

Suppose $ a_1 , a_2 , j \in \mathbb{ Z } $ satisfy $ a_1 + a_2 \in 2 \mathbb{ Z } $, and let $ f_1 $ and $ f_2 $ be as in Lemma \ref{lem:six-point-equation}.
Then two functions $ p_1 , p_2 \colon \mathbb{ Z } \to \mathbb{ C } $ defined as 
\begin{align}
    p_1 ( u ) \coloneqq f_1 ( j u + a_1 ), & &
    p_2 ( u ) \coloneqq f_2 ( j u + a_2 )
\end{align}
satisfy the complex Kac--Bernstein functional equation 
\begin{align}
  & \phantom{ {} = {} }  p_1 ( u + v ) p_2 ( u - v ) p_1 ( u' + v' ) p_2 ( u' - v' ) \\
  & = p_1 ( u + v' ) p_2 ( u - v' ) p_1 ( u' + v ) p_2 ( u' - v ) \label{eq:scaling-main}
\end{align}
for any $ u , u' , v , v' \in \mathbb{ Z } $.
    
\end{lemma}

\begin{proof}

Since $ a_1 + a_2 \in 2 \mathbb{ Z } $, the values
\begin{align}
    \Tilde{ u } \coloneqq j u + \frac{ a_1 + a_2 }{ 2 }, & &
    \Tilde{ v } \coloneqq j v + \frac{ a_1 - a_2 }{ 2 }, & &
    \Tilde{ u }' \coloneqq j u' + \frac{ a_1 + a_2 }{ 2 }, & &
    \Tilde{ v }' \coloneqq j v' + \frac{ a_1 - a_2 }{ 2 }
\end{align}
are all integers.
We have
\begin{align}
    p_1 ( u + v ) p_2 ( u - v ) p_1 ( u' + v' ) p_2 ( u' - v' )
    & = f_1 ( \Tilde{ u } + \Tilde{ v } ) f_2 ( \Tilde{ u } - \Tilde{ v } ) f_1 ( \Tilde{ u }' + \Tilde{ v }' ) f_2 ( \Tilde{ u }' - \Tilde{ v }' ), \\
    p_1 ( u + v' ) p_2 ( u - v' ) p_1 ( u' + v ) p_2 ( u' - v )
    & = f_1 ( \Tilde{ u } + \Tilde{ v }' ) f_2 ( \Tilde{ u } - \Tilde{ v }' ) f_1 ( \Tilde{ u }' + \Tilde{ v } ) f_2 ( \Tilde{ u }' - \Tilde{ v } )
\end{align}
by the definition of $ p_1 $ and $ p_2 $ and hence \eqref{eq:scaling-main} holds by \eqref{eq:complex-Kac-Bernstein}.
\end{proof}

If there exists $ a_1 \in \supp f_1 $ such that $ S_2 ( a_1 , 2 ) $ is sufficiently large, then one of the conditions of Theorem \ref{thm:Z-characterize} holds as follows:

\begin{lemma}
    \label{lem:supp-large-Gauss}

Let $ f_1$, $ f_2 $, $ S_1$, and $ S_2 $ be as in Lemma \ref{lem:supp-arithmetic-progression}, and $ a_1 \in \supp f_1 $.

\begin{enumerate}
    \item \label{item:supp-large-Gauss-large}

If $ \# S_2 ( a_1 , 2 ) \geq 3 $, then the condition \ref{item:Z-characterize-Gauss} of Theorem \ref{thm:Z-characterize} holds.

\item \label{item:supp-large-Gauss-number-2}

If $ \# S_2 ( a_1 , 2 ) = 2 $, then the condition \ref{item:Z-characterize-six-point} of Theorem \ref{thm:Z-characterize} holds.
    
\end{enumerate}

\end{lemma}

\begin{proof}

\begin{enumerate}
    \item 

    First, we prove that there exist $ a_2 \in \supp f_2 $ and $ j_0 \in 1 + 2 \mathbb{ Z } $ such that
    \begin{align}
        a_1 + j_0 \mathbb{ Z } = \supp f_1, & &
        a_2 + j_0 \mathbb{ Z } = \supp f_2. \label{eq:supp-large-Gauss-large-proof-supp}
    \end{align}
    Since we have
    \begin{align}
        \# S_2 ( a_1 , 4 ) + \# S_2 ( a_1 + 2 , 4 ) = \# S_2 ( a_1 , 2 ) \geq 3,
    \end{align}
    at least one of $ \# S_2 ( a_1 , 4 ) \geq 2 $ and $ \# S_2 ( a_1 + 2 , 4 ) \geq 2 $ holds.
Thus, there exist $ a_2 , j \in \mathbb{ Z } $ with $ a_2 \pm 2 j \in \supp f_2 $ such that $ a_1 \in S_1 ( a_2 \pm 2 j , 2 ) $ holds.
We have
\begin{align}
    a_1 + j \mathbb{ Z } = S_1 ( a_1 , j ) , & &
    a_2 + j \mathbb{ Z } = S_2 ( a_2 , j )
\end{align}
by Lemma \ref{lem:supp-arithmetic-progression} \ref{item:supp-arithmetic-progression-main} and hence
\begin{align}
    j_0 
    \coloneqq \min \left\{ j\in\mathbb Z_{\geq1} \; \middle| \;
  \begin{aligned}
  & \text{there exist } b_1,b_2\in\mathbb Z
  \text{ with } b_1+b_2\in2\mathbb Z,\\
  & b_1+j\mathbb Z=S_1(b_1,j), \quad
  b_2+j\mathbb Z=S_2(b_2,j)
  \end{aligned}
\right\} \label{eq:supp-large-Gauss-large-proof-j0-minimum}
\end{align}
can be defined.
We have $ j_0 \in 1 + 2 \mathbb{ Z } $ by Lemma \ref{lem:supp-arithmetic-progression} \ref{item:supp-arithmetic-progression-main} and the minimality of $ j_0 $.
There exist $ b_1 , b_2 \in \mathbb{ Z } $ such that
\begin{align}
    b_1 + j_0 \mathbb{ Z } = S_1 ( b_1 , j_0 ), & &
    b_2 + j_0 \mathbb{ Z } = S_2 ( b_2 , j_0 ) \label{eq:supp-large-Gauss-large-proof-b1-b2-define}
\end{align}
by \eqref{eq:supp-large-Gauss-large-proof-j0-minimum}.

Next, we prove $ b_1 + j_0 \mathbb{ Z } = \supp f_1 $.
Since
\begin{align}
    b_1 + j_0 \mathbb{ Z } = S_1 ( b_1 , j_0 ) \subset \supp f_1
\end{align}
holds by \eqref{eq:supp-arithmetic-progression-S-define}, it suffices to show $ b_1' \in b_1 + j_0 \mathbb{ Z } $ for any $ b_1' \in \supp f_1 $.
We may assume
\begin{align}
    b_1 + b_1' \in 2 \mathbb{ Z } , & &
    b_1 \leq b_1' < b_1 + 2 j_0
\end{align}
by replacing $ b_1 $ if necessary.
By $ b_1 + b_1' \in 2 \mathbb{ Z } $ and $ j_0 \in 1 + 2 \mathbb{ Z } $, there exists $ j' \in \mathbb{ Z } $ such that either $ b_1' - b_1 = 4 j' $ or $ b_1 + 2 j_0 - b_1' = 4 j' $ holds.
Then we have $ 0 \leq j' < j_0 $.
Since $ S_2 ( b_1' , 2 ) \cap \{ b_2 , b_2 + j_0 \} \neq \emptyset $ holds by \eqref{eq:supp-large-Gauss-large-proof-b1-b2-define} and $ j_0 \in 1 + 2 \mathbb{ Z } $, we have $ j' = 0 $ by Lemma \ref{lem:supp-arithmetic-progression} \ref{item:supp-arithmetic-progression-main} and the minimality of $ j_0 $.
Thus, we have $ b_1' = b_1 \in b_1 + j_0 \mathbb{ Z } $ and hence $ b_1 + j_0 \mathbb{ Z } = \supp f_1 $.
Similarly, we obtain $ a_2 + j_0 \mathbb{ Z } = \supp f_2 $ and hence \eqref{eq:supp-large-Gauss-large-proof-supp}.

Since $ j_0 \in 1 + 2 \mathbb{ Z } $ holds, we may assume $ a_1 + a_2 \in 2 \mathbb{ Z } $ by replacing $ a_2 $ with $ a_2 + j_0 $ if necessary.
We define functions $ p_1 , p_2 \colon \mathbb{ Z } \to \mathbb{ C } $ by
\begin{align}
p_1 ( u ) \coloneqq f_1 ( a_1 + u j_0 ), & &
p_2 ( u ) \coloneqq f_2 ( a_2 + u j_0 ).
\end{align}
By Lemma \ref{lem:scaling}, the functions $ p_1 $ and $ p_2 $ satisfy \eqref{eq:scaling-main}.
Since \eqref{eq:supp-large-Gauss-large-proof-supp} implies that
\begin{align}
\supp p_1 = \supp p_2 = \mathbb{ Z },
\end{align}
Lemma \ref{lem:supp-all} $ \text{ \ref{item:supp-all-Kac-Bernstein} } \Longrightarrow \text{ \ref{item:supp-all-Gauss} } $ shows that the condition \ref{item:Z-characterize-Gauss} of Theorem \ref{thm:Z-characterize} holds.

\item

There exist $ a_2 \in \mathbb{ Z } $ and $ k \in \mathbb{ Z } \setminus \{ 0 \} $ such that $ S_2 ( a_1 , 2 ) = \{ a_2 \pm k \} $ holds by $ \# S_2 ( a_1 , 2 ) = 2 $.
Then we have $ a_1 + a_2 + k \in 2 \mathbb{ Z } $ and hence \eqref{eq:Z-characterize-six-point-equation} by Lemma \ref{lem:six-point-equation}.
In addition, we get
\begin{align}
    a_1 \in S_1 ( a_2 + k , 2 ), & &\{ a_2 \pm k \} = S_2 ( a_1 , 2 ) \subset \supp f_2
\end{align}
and hence \eqref{eq:supp-arithmetic-progression-supp-include} by Lemma \ref{lem:supp-arithmetic-progression} \ref{item:supp-arithmetic-progression-supp}.
Since
\begin{align}
    a_2 \notin \{ a_2 \pm k \} = S_2 ( a_1 , 2 )
\end{align}
holds, we have $ a_1 + a_2 , k \in 1 + 2 \mathbb{ Z } $ by \eqref{eq:supp-arithmetic-progression-supp-include}.

We now prove $ \{ a_2 , a_2 \pm k \} = \supp f_2 $.
We have $ \{ a_2 , a_2 \pm k \} \subset \supp f_2 $ by $ S_2 ( a_1 , 2 ) = \{ a_2 \pm k \} $ and \eqref{eq:supp-arithmetic-progression-supp-include}.
Thus, it suffices to show $ a_2 = a_2' $ for any $ a_2' \in \supp f_2 $ with $ a_2' \neq a_2 \pm k $.
Since $ a_2' \in S_2 ( a_1 + k , 2 ) $ holds by $ k \in 1 + 2 \mathbb{ Z } $ and 
\begin{align}
    a_2' \notin \{ a_2 \pm k \} = S_2 ( a_1 , 2 ),
\end{align}
we have
\begin{align}
    a_2' \pm k \in S_2 ( a_1 , 2 ) = \{ a_2 \pm k \}
\end{align}
by \eqref{eq:supp-arithmetic-progression-supp-include} and Lemma \ref{lem:supp-arithmetic-progression} \ref{item:supp-arithmetic-progression-supp}.
Thus, we obtain $ a_2 = a_2' $.

Finally, it suffices to show $ \# S_1 ( a_2 , 2 ) = 2 $, since $ \{ a_1 , a_1 \pm k \} = \supp f_1 $ follows by a similar argument.
The condition \ref{item:Z-characterize-Gauss} of Theorem \ref{thm:Z-characterize} does not hold by $ \# S_2 ( a_1 , 2 ) = 2 $ and hence we have $ \# S_1 ( a_2 , 2 ) \leq 2 $ by Lemma \ref{lem:supp-large-Gauss} \ref{item:supp-large-Gauss-large}.
Since
\begin{align}
   \# S_1 ( a_2 , 2 ) \geq \# \{ a_1 \pm k \} = 2 
\end{align}
holds by $ S_2 ( a_1 , 2 ) = \{ a_2 \pm k \} $, we obtain $ \# S_1 ( a_2 , 2 ) = 2 $.
\qedhere

\end{enumerate}

\end{proof}

\subsection{Completion of the proof of Theorem \ref{thm:Z-characterize}}
\label{subsec:Z-characterize-proof-complete}

In this subsection, we complete the proof of Theorem \ref{thm:Z-characterize} by using Lemmas \ref{lem:six-point-equation} and \ref{lem:supp-large-Gauss}.

\begin{proof}[Proof of Theorem \ref{thm:Z-characterize}]

It suffices to show the conclusion when
\begin{align}
    ( \# S_1 ( a , 2 ) , \# S_2 ( a , 2 ) ) = ( 0 , * ) , ( * , 0 ) , ( 1 , 1 )
\end{align}
holds for $ a = 0 , 1 $ by Lemma \ref{lem:supp-large-Gauss}.
Here, the symbol $ * $ denotes any nonnegative integer or infinity.
In this case, we prove that these conditions correspond to those of Theorem \ref{thm:Z-characterize}, as shown in Table \ref{tab:Z-characterize-table}.

\begin{table}[ht]
            \centering
            \caption{Correspondence table of the conditions of $ ( \# S_1 , \# S_2 ) $ and the conditions of Theorem \ref{thm:Z-characterize}}
            \label{tab:Z-characterize-table}

\vspace{8pt}
            
\begin{tabular}{c|c||c|c|c}
\multicolumn{2}{c||}{\multirow{2}{*}{ }} & \multicolumn{3}{c}{ $ ( \# S_1 ( 0 , 2 ) , \# S_2 ( 0 , 2 ) ) $ } \\
\cline{3-5}
\multicolumn{2}{c||}{ } & $ ( 0 , * ) $ & $ ( * , 0 ) $ & $ ( 1 , 1 ) $ \\
\hline \hline
\multirow{3}{*}{ $ ( \# S_1 ( 1 , 2 ) , \# S_2 ( 1 , 2 ) ) $ } & $ ( 0 , * ) $ & \ref{item:Z-characterize-f1-zero} & \ref{item:Z-characterize-independent} & \ref{item:Z-characterize-f1-one-point} \\
\cline{2-5}
& $ ( * , 0 ) $ & \ref{item:Z-characterize-independent} & \ref{item:Z-characterize-f2-zero} & \ref{item:Z-characterize-f2-one-point} \\
\cline{2-5}
& $ ( 1 , 1 ) $ & \ref{item:Z-characterize-f1-one-point} & \ref{item:Z-characterize-f2-one-point} & \ref{item:Z-characterize-six-point}
\end{tabular}
\end{table}

Except for the case of
\begin{align}
    ( \# S_1 ( 0 , 2 ) , \# S_2 ( 0 , 2 ) ) = ( \# S_1 ( 1 , 2 ) , \# S_2 ( 1 , 2 ) ) = ( 1 , 1 ), \label{eq:Z-characterize-proof-non-trivial}
\end{align}
it is clear from the definitions that the correspondence is as shown
in Table \ref{tab:Z-characterize-table}.

Thus, it suffices to show \ref{item:Z-characterize-six-point} of Theorem \ref{thm:Z-characterize} when \eqref{eq:Z-characterize-proof-non-trivial} holds.
In this case, there exist $ a_1, \xi \in 2 \mathbb{ Z } $ and $ a_2 , k \in 1 + 2 \mathbb{ Z } $ such that
\begin{align}
    S_1 ( 0 , 2 ) & = \{ a_1 \}, &
    S_1 ( 1 , 2 ) & = \{ a_1 + k \}, \\
    S_2 ( 0 , 2 ) & = \{ \xi \}, &
    S_2 ( 1 , 2 ) & = \{ a_2 \} \label{eq:Z-characterize-proof-represent}
\end{align}
hold.
Since
\begin{align}
    u \coloneqq \frac{ a_1 + \xi }{ 2 }, & &
    v \coloneqq \frac{ a_1 - \xi }{ 2 }, & &
    u' \coloneqq \frac{ a_1 + a_2 + k }{ 2 }, & &
    v' \coloneqq \frac{ a_1 - a_2 + k }{ 2 }
\end{align}
are all integers, we have
\begin{align}
  & \phantom{ { } = { } } f_1 ( u + v' ) f_2 ( u - v' ) f_1 ( u' + v ) f_2 ( u' - v ) \\
  & = f_1 ( u + v ) f_2 ( u - v ) f_1 ( u' + v' ) f_2 ( u' - v' ) \\
  & = f_1 ( a_1 ) f_2 ( \xi ) f_1 ( a_1 + k ) f_2 ( a_2 ) \label{eq:Z-characterize-proof-apply-Kac-Bernstein}
\end{align}
by \eqref{eq:complex-Kac-Bernstein}.
Since $ f_1 ( a_1 ) f_2 ( \xi ) f_1 ( a_1 + k ) f_2 ( a_2 ) \neq 0 $ holds by \eqref{eq:Z-characterize-proof-represent}, we have
\begin{align}
    a_1 + \frac{ \xi - a_2 + k }{ 2 } = u + v' \in \supp f_1 = \{ a_1 , a_1 + k \}
\end{align}
by \eqref{eq:Z-characterize-proof-apply-Kac-Bernstein}.
Thus, we get $ \xi = a_2 \pm k $ and hence \eqref{eq:Z-characterize-six-point-support} follows from \eqref{eq:Z-characterize-proof-represent}.
In addition, we have \eqref{eq:six-point-equation-assume} and hence \eqref{eq:Z-characterize-six-point-equation} by Lemma \ref{lem:six-point-equation}.
Thus, the condition \ref{item:Z-characterize-six-point} of Theorem \ref{thm:Z-characterize} holds.
\end{proof}

\section{Proof of Corollary \ref{cor:R-characterize}}
\label{sec:R-characterize-proof}

In this section, we prove Corollary \ref{cor:R-characterize} by using Theorem \ref{thm:Z-characterize}.

\begin{lemma}
\label{lem:Gauss-extension-existence}

Suppose that $ q_1 , q_2 , q_3 , q_4 , q_5 , q_6 \in \mathbb{ C } $ satisfy
\begin{align}
    | q_1 - q_2 | , | q_3 - q_2 | ,
    | q_4 - q_5 | , | q_6 - q_5 | < \frac{1}{12}.
\end{align}
Define $ c_1 , c_2 , c_3 , c_4 , c_5 , c_6 \in \mathbb{ C } $ by \eqref{eq:supp-all-Gauss-const-def}, and let
\begin{align}
    g_1 ( u ) &\coloneqq e^{ c_1 u^2 + c_2 + c_3 u + c_4 }, &
    g_2 ( u ) &\coloneqq e^{ c_1 u^2 - c_2 + c_5 u + c_6 }.
\end{align}
Then, for any $ - 1 \leq u \leq 1 $, we have
\begin{align}
    \left| \frac{ g_1 ( u ) }{ g_1 ( 0 ) } - 1 \right| < \frac{1}{5}, & &
    \left| \frac{ g_2 ( u ) }{ g_2 ( 0 ) } - 1 \right| < \frac{1}{5}. \label{eq:Gauss-extension-existence-bound}
\end{align}
\end{lemma}

\begin{proof}

We have
\begin{align}
    c_1
    &= \frac{ q_1 - 2 q_2 + q_3 + q_4 - 2 q_5 + q_6 }{ 4 }, &
    c_3 &= \frac{ q_3 - q_1 }{ 2 }, &
    c_5 &= \frac{ q_6 - q_4 }{ 2 }
\end{align}
by \eqref{eq:supp-all-Gauss-const-def}.
Thus, the assumption implies
\begin{align}
    | c_1 | , | c_3 | , | c_5 | < \frac{1}{12},
\end{align}
and hence we have
\begin{align}
    | c_1 u^2 + c_3 u |
    &\leq | c_1 | + | c_3 | < \frac{1}{6}, &
    | c_1 u^2 + c_5 u |
    &\leq | c_1 | + | c_5 | < \frac{1}{6}
\end{align}
for any $ | u | \leq 1 $.
Therefore, the inequalities
\begin{align}
    \left| \frac{ g_1 ( u ) }{ g_1 ( 0 ) } - 1 \right|
    &\leq e^{ | c_1 u^2 + c_3 u | } - 1
    < e^{ 1 / 6 } - 1 < \frac{1}{5}, &
    \left| \frac{ g_2 ( u ) }{ g_2 ( 0 ) } - 1 \right|
    &< \frac{1}{5}
\end{align}
are obtained.
\end{proof}

\begin{lemma}
\label{lem:Gauss-extension-uniqueness}

Let $ d_1 , d_2 , d_3 , d_1' , d_2' , d_3' \in \mathbb{ C } $, and suppose that the functions
\begin{align}
    f ( u ) &\coloneqq e^{ d_1 u^2 + d_2 u + d_3 }, &
    g ( u ) &\coloneqq e^{ d_1' u^2 + d_2' u + d_3' }
\end{align}
satisfy
\begin{align}
    f ( u ) = g ( u ) \label{eq:Gauss-extension-uniqueness-even}
\end{align}
for any $ u \in 2 \mathbb{ Z } $.
If
\begin{align}
    \left| \frac{ f ( 1 ) }{ g ( 1 ) } - 1 \right| < \frac{1}{2}
\end{align}
holds, then one has \eqref{eq:Gauss-extension-uniqueness-even} for any $ u \in \mathbb{ Z } $.
\end{lemma}

\begin{proof}

First, we show that $ f ( 1 ) = g ( 1 ) $.
Let $ F ( u ) \coloneqq f ( u ) / g ( u ) $.
It suffices to prove that $ F ( 1 ) = 1 $.
Since $ F ( 0 ) = 1 $, there exist $ a , b \in \mathbb{ C } $ such that $ F ( u ) = e^{ a u^2 + b u } $.
Then
\begin{align}
    F ( 1 )^8
    = e^{ 8 ( a + b ) }
    = e^{ 3 ( 4a + 2b ) - ( 4a - 2b ) }
    = \frac{ F ( 2 )^3 }{ F ( - 2 ) }
    = 1.
\end{align}
Since
\begin{align}
    | F ( 1 ) - 1 | < \frac{1}{2}
    < 2 \sin \frac{\pi}{8}
\end{align}
holds, we obtain $ F ( 1 ) = 1 $.

Moreover, we have
\begin{align}
    au^2 + bu
    = ( a + b ) ( 2u - u^2 ) + ( 4a + 2b ) \frac{u ( u - 1 )}{2}
\end{align}
for any $ u \in \mathbb{ Z } $, and hence
\begin{align}
    F ( u )
    = e^{ au^2 + bu }
    = F ( 1 )^{ 2u - u^2 } F ( 2 )^{ u ( u - 1 ) / 2 }
    = 1
\end{align}
is obtained.
\end{proof}

\begin{proof}[Proof of Corollary \ref{cor:R-characterize}]

It suffices to prove \ref{item:R-characterize-Gauss} in Corollary \ref{cor:R-characterize} when $ f_1 , f_2 \not\equiv 0 $.
By the continuity of $ f_1 $ and $ f_2 $, for every sufficiently small $ \epsilon > 0 $, there exist $ N_1 , N_2 \in \mathbb{ Z } $ such that
\begin{align}
    f_1 ( ( N_1 + m ) \epsilon ) \neq 0, & &
    f_2 ( ( N_2 + m ) \epsilon ) \neq 0
\end{align}
for any $ m = 0 , 1 , 2 , 3 $.
By applying Theorem \ref{thm:Z-characterize} to $ u \mapsto f_1 ( u \epsilon ) $ and $ u \mapsto f_2 ( u \epsilon ) $, we obtain $ d_1 , d_2 , d_3 , d_4 , d_5 \in \mathbb{ C } $ such that \eqref{eq:R-characterize-Gauss-equation} holds for any $ u \in 2 \epsilon \mathbb{ Z } $.
In particular, $ f_1 ( 0 ) , f_2 ( 0 ) \neq 0 $.

By taking $ \epsilon $ smaller if necessary, we may choose $ q_1 , q_2 , q_3 , q_4 , q_5 , q_6 \in \mathbb{ C } $ satisfying
\begin{align}
    | q_1 - q_2 | , | q_3 - q_2 | ,
    | q_4 - q_5 | , | q_6 - q_5 | < \frac{1}{12}
\end{align}
such that
\begin{align}
    f_1 ( - \epsilon ) = e^{ q_1 }, & &
    f_1 ( 0 ) = e^{ q_2 }, & &
    f_1 ( \epsilon ) = e^{ q_3 }, & &
    f_2 ( - \epsilon ) = e^{ q_4 }, & &
    f_2 ( 0 ) = e^{ q_5 }, & &
    f_2 ( \epsilon ) = e^{ q_6 },
\end{align}
and
\begin{align}
    \left| \frac{ f_1 ( u ) }{ f_1 ( 0 ) } - 1 \right| < \frac{1}{5}, & &
    \left| \frac{ f_2 ( u ) }{ f_2 ( 0 ) } - 1 \right| < \frac{1}{5} \label{eq:R-characterize-proof-f-bound}
\end{align}
for any $ - \epsilon \leq u \leq \epsilon $.
We define $ c_1 , c_2 , c_3 , c_4 , c_5 , c_6 \in \mathbb{ C } $ by \eqref{eq:supp-all-Gauss-const-def}, and let
\begin{align}
    g_1 ( u )
    &\coloneqq e^{ c_1 u^2 / \epsilon^2 + c_2 + c_3 u / \epsilon + c_4 }, &
    g_2 ( u )
    &\coloneqq e^{ c_1 u^2 / \epsilon^2 - c_2 + c_5 u / \epsilon + c_6 }. \label{eq:R-characterize-proof-g-definition}
\end{align}
Then we have
\begin{align}
    f_1 ( u )
    &= g_1 ( u ), &
    f_2 ( u )
    &= g_2 ( u ) \label{eq:R-characterize-proof-even-equality}
\end{align}
for any $ u \in 2 \epsilon \mathbb{ Z } $ by Lemma \ref{lem:supp-all}.
Combining \eqref{eq:R-characterize-proof-f-bound} with
Lemma \ref{lem:Gauss-extension-existence}, we obtain
\begin{align}
    \left| \frac{ f_1 ( \epsilon ) }{ g_1 ( \epsilon ) } - 1 \right| < \frac{1}{2}, & &
    \left| \frac{ f_2 ( \epsilon ) }{ g_2 ( \epsilon ) } - 1 \right| < \frac{1}{2}.
\end{align}
Then $ f_1 $ and $ f_2 $ can be expressed in the form \eqref{eq:R-characterize-Gauss-equation} on $ \epsilon \mathbb{ Z } $ by applying Lemma \ref{lem:supp-all} to $ u \mapsto f_1 ( u \epsilon / 2 ) $ and $ u \mapsto f_2 ( u \epsilon / 2 ) $.
Thus, we have \eqref{eq:R-characterize-proof-even-equality} for any $ u \in \epsilon \mathbb{ Z } $ by Lemma \ref{lem:Gauss-extension-uniqueness}.

Similarly, for $ M = 1 , 2 , 3 , \cdots $, applying Lemma \ref{lem:supp-all} with its $ f_1 $ and $ f_2 $ replaced by $ n \mapsto f_1 ( n \epsilon / 2^M ) $ and $ n \mapsto f_2 ( n \epsilon / 2^M ) $, respectively, and repeating the above argument, we have \eqref{eq:R-characterize-proof-even-equality} for any $ u \in ( \epsilon / 2^{ M - 1 } ) \mathbb{ Z } $.
Since $ f_1 $ and $ f_2 $ are continuous, \eqref{eq:R-characterize-Gauss-equation} holds for any $ u \in \mathbb{ R } $.
Thus, \ref{item:R-characterize-Gauss} in Corollary \ref{cor:R-characterize} holds.
\end{proof}

\section{Proof of the results in Sections \ref{sec:main-relate} and \ref{sec:Kac-Bernstein-probability}}
\label{sec:in-main-relate-proof}

In this section, we prove the results in Sections \ref{sec:main-relate} and \ref{sec:Kac-Bernstein-probability}.
In Subsections \ref{subsec:Kac-Bernstein-Fourier}, \ref{subsec:complex-T}, \ref{subsec:Kac-Bernstein-complex-R}, and \ref{subsec:Baryshnikov-Eisenberg-Stadje}, we prove Proposition \ref{prop:Kac-Bernstein-Fourier}, Corollary \ref{cor:complex-T}, Corollary \ref{cor:Kac-Bernstein-complex-R}, and Corollary \ref{cor:Baryshnikov-Eisenberg-Stadje}, respectively.

\subsection{Proof of Proposition \ref{prop:Kac-Bernstein-Fourier}}
\label{subsec:Kac-Bernstein-Fourier}

In this subsection, we prove Proposition \ref{prop:Kac-Bernstein-Fourier}.
First, we show the following lemma:

\begin{lemma}
    \label{lem:pushforward-Fourier}

Let $ X $ and $ F $ be as in Problem \ref{Q:Kac-Bernstein-complex}, and $ \mu_1 $ and $ \mu_2 $ be complex Borel measures on $ X $.
Then
\begin{align}
    ( F_* ( \mu_1 \times \mu_2 ) )^{ \hat{ } } ( u , v ) = \hat{ \mu_1 } ( u + v ) \hat{ \mu_2 } ( u - v )
\end{align}
holds for any $ u , v \in \hat{ X } $.

\end{lemma}

\begin{proof}

Put
\begin{align}
    ( X_1 , X_2 ) 
    \coloneqq F ( x_1 , x_2 )
    = ( x_1 + x_2 , x_1 - x_2 ).
\end{align}
Then we have
\begin{align}
    ( F_* ( \mu_1 \times \mu_2 ) )^{ \hat{ } } ( u , v )
    & = \int_{ X \times X } u ( X_1 ) v ( X_2 ) dF_* ( \mu_1 \times \mu_2 ) ( X_1 , X_2 ) \\
    & = \int_{ X } \int_{ X } u ( x_1 + x_2 ) v ( x_1 - x_2 ) d \mu_1 ( x_1 ) d \mu_2 ( x_2 ) \\
    & = \int_{ X } ( u + v ) ( x_1 ) d \mu_1 ( x_1 ) \int_{ X } ( u - v ) ( x_2 ) d \mu_2 ( x_2 ) \\
    & = \hat{ \mu_1 } ( u + v ) \hat{ \mu_2 } ( u - v )
\end{align}
by the definition of the pushforward measure.
\end{proof}

Lemma \ref{lem:pushforward-Fourier} implies Proposition \ref{prop:Kac-Bernstein-Fourier} as follows:

\begin{proof}[Proof of Proposition \ref{prop:Kac-Bernstein-Fourier}]

The condition \ref{item:Kac-Bernstein-Fourier-independent} holds if and only if
\begin{align}
    ( F_* ( \mu_1 \times \mu_2 ) )^{ \hat{ } } ( u , v ) ( F_* ( \mu_1 \times \mu_2 ) )^{ \hat{ } } ( u' , v' )
    = ( F_* ( \mu_1 \times \mu_2 ) )^{ \hat{ } } ( u , v' ) ( F_* ( \mu_1 \times \mu_2 ) )^{ \hat{ } } ( u' , v )
\end{align}
for any $ u , u' , v , v' \in \hat{ X } $.
This condition is equivalent to \ref{item:Kac-Bernstein-Fourier-equation} by Lemma \ref{lem:pushforward-Fourier}.
\end{proof}

\subsection{Proof of Corollary \ref{cor:complex-T}}
\label{subsec:complex-T}

In this subsection, we prove Corollary \ref{cor:complex-T}.
First, we show the following lemma:

\begin{lemma}
\label{lem:key-complex-T}

Let $ c_2 \in \mathbb{ C } $ and $ u \in \mathbb{ Z } $.

\begin{enumerate}
    \item \label{item:key-complex-T-hyperbolic}

One has $ \cosh ( c_2 ) + \sinh ( c_2 ) ( - 1 )^u = e^{ c_2 ( - 1 )^u } $.

\item \label{item:key-complex-T-scaling}

Let $ k = 1 , 3 , 5 , \cdots $ and $ \theta \in \mathbb{ T } $.
Then
\begin{align}
    \int_{ \mathbb{ T } } e^{ - i u x } d \rho ( k x + \theta )
    =
    \left\{
    \begin{aligned}
       & e^{ i u \theta / k } \hat{ \rho } \left( \frac{ u }{ k } \right)  & & \text{ if $ u \in k \mathbb{ Z } $} \\
       & 0 & & \text{ if $ u \in \mathbb{ Z } \setminus k \mathbb{ Z } $}
    \end{aligned}
    \right. \label{eq:key-complex-T-scaling-conclude}
\end{align}
holds for any complex Borel measure $ \rho $ on $ \mathbb{ T } $.

\item \label{item:key-complex-T-main}

Let $ k $ and $ \rho $ be as in \ref{item:key-complex-T-scaling}, and $ a \in \mathbb{ Z } $ and $ c_4 \in \mathbb{ C } $.
Then
\begin{align}
    \mu ( x ) \coloneqq e^{ i a x + c_4 } ( \cosh ( c_2 ) \rho ( k x ) + \sinh ( c_2 ) \rho ( k x + \pi ) ) \label{eq:key-complex-T-main-mu-definition}
\end{align}
satisfies
\begin{align}
    \hat{ \mu } ( u )
    =
  \left\{
  \begin{aligned}
    & e^{ c_2 ( - 1 )^{ u - a } + c_4 } \hat{ \rho } \left( \frac{ u - a }{ k } \right) & & \text{ if $ u \in a + k \mathbb{ Z } $} \\
    & 0 & & \text{ if $ u \notin a + k \mathbb{ Z } $ }
  \end{aligned}
  \right. . \label{eq:key-complex-T-main-conclude}
\end{align}
    
\end{enumerate}

\end{lemma}

\begin{proof}

\begin{enumerate}
    \item 
    We obtain
\begin{align}
    \cosh ( c_2 ) + \sinh ( c_2 ) ( - 1 )^u
    & = \frac{ e^{ c_2 } + e^{ - c_2 } }{ 2 } + \frac{ ( e^{ c_2 } - e^{ - c_2 } ) ( - 1 )^u }{ 2 } \\
    & = \frac{ 1 }{ 2 } ( ( 1 + ( - 1 )^u ) e^{ c_2 } + ( 1 - ( - 1 )^u ) e^{ - c_2 } ) \\
    & = e^{ c_2 ( - 1 )^u }
\end{align}
by the definition of the hyperbolic functions.

\item
We have
\begin{align}
    \int_{ \mathbb{ T } } e^{ - i u x } d \rho ( k x + \theta )
    = \frac{ 1 }{ k } \int_{ \mathbb{ T } } \sum_{ x \in \mathbb{ T }, y = k x + \theta } e^{ - i u x } d \rho ( y ) \label{eq:measure-scaling-Fourier}
\end{align}
by \eqref{eq:complex-T-boundary-scaling}.
Then
\begin{align}
    \frac{ 1 }{ k } \sum_{ x \in \mathbb{ T }, y = k x + \theta } e^{ - i u x }
    = \left\{ 
    \begin{aligned}
        & e^{ - i u ( y - \theta ) / k } & & \text{ if $ u \in k \mathbb{ Z } $ } \\
        & 0 & & \text{ if $ u \in \mathbb{ Z } \setminus k \mathbb{ Z } $ }
    \end{aligned}
    \right.
\end{align}
holds and hence we have \eqref{eq:key-complex-T-scaling-conclude} for any $ u \in \mathbb{ Z } \setminus k \mathbb{ Z } $.
When $ u \in k \mathbb{ Z } $, we have
\begin{align}
   \frac{ 1 }{ k } \int_{ \mathbb{ T } } \sum_{ x \in \mathbb{ T }, y = k x + \theta } e^{ - i u x } d \rho ( y )
   & = \int_{ \mathbb{ T } } e^{ - i u ( y - \theta ) / k } d \rho ( y ) \\
   & = e^{ i u \theta / k } \int_{ \mathbb{ T } } e^{ - i u y / k } d \rho ( y ) \\
   & = e^{ i u \theta / k } \hat{ \rho } \left( \frac{ u }{ k } \right) \label{eq:measure-scaling-conclude}
\end{align}
and hence \eqref{eq:key-complex-T-scaling-conclude} follows from \eqref{eq:measure-scaling-Fourier} and \eqref{eq:measure-scaling-conclude}.

\item

We have
\begin{align}
    \hat{ \mu } ( u )
    & = \int_{ \mathbb{ T } } e^{ - i u x } d \mu ( x ) \\
    & = e^{ c_4 } \left( \cosh ( c_2 ) \int_{ \mathbb{ T } } e^{ - i x ( u - a ) } d \rho ( k x ) + \sinh ( c_2 ) \int_{ \mathbb{ T } } e^{ - i x ( u - a ) } d \rho ( k x + \pi ) \right) \label{eq:mu-Fourier-transform}
\end{align}
by \eqref{eq:key-complex-T-main-mu-definition}.
Thus, when $ u \notin a + k \mathbb{ Z } $, we obtain $ \hat{ \mu } ( u ) = 0 $ by \ref{item:key-complex-T-scaling}.

When $ u \in a + k \mathbb{ Z } $, we have
\begin{align}
    \int_{ \mathbb{ T } } e^{ - i x ( u - a ) } d \rho ( k x )
    & = \hat{ \rho } \left( \frac{ u - a }{ k } \right), &
    \int_{ \mathbb{ T } } e^{ - i x ( u - a ) } d \rho ( k x + \pi )
    = e^{ i ( u - a ) \pi / k } \hat{ \rho } \left( \frac{ u - a }{ k } \right)
\end{align}
by \ref{item:key-complex-T-scaling}.
Then $ e^{ i ( u - a ) \pi / k } = ( - 1 )^{ u - a } $ holds by $ k \in 1 + 2 \mathbb{ Z } $ and hence we have
\begin{align}
    & \phantom{ { } = { } }\cosh ( c_2 ) \int_{ \mathbb{ T } } e^{ - i x ( u - a ) } d \rho ( k x ) + \sinh ( c_2 ) \int_{ \mathbb{ T } } e^{ - i x ( u - a ) } d \rho ( k x + \pi ) \\
    & = ( \cosh ( c_2 ) + ( - 1 )^{ u - a } \sinh ( c_2 ) ) \hat{ \rho } \left( \frac{ u - a }{ k } \right). \label{eq:key-complex-T-scaling-apply}
\end{align}
Since $ \cosh ( c_2 ) + ( - 1 )^{ u - a } \sinh ( c_2 ) = e^{ c_2 ( - 1 )^{ u - a } } $ holds by \ref{item:key-complex-T-hyperbolic}, we obtain \eqref{eq:key-complex-T-main-conclude} by \eqref{eq:mu-Fourier-transform} and \eqref{eq:key-complex-T-scaling-apply}.
\qedhere
\end{enumerate}
    
\end{proof}

By Theorem \ref{thm:Z-characterize}, Proposition \ref{prop:Kac-Bernstein-Fourier}, and Lemma \ref{lem:key-complex-T}, we obtain Corollary \ref{cor:complex-T} as follows:

\begin{proof}[Proof of Corollary \ref{cor:complex-T}]

For any complex Borel measure $ \mu $ on $ \mathbb{ T } $, the conditions $ \mu \in M_{ + } $ and $ \mu \in M_{ - } $ are equivalent to $ \supp \hat{ \mu } \subset 2 \mathbb{ Z } $ and $ \supp \hat{ \mu } \subset 1 + 2 \mathbb{ Z } $, respectively.
Thus, the conditions of Theorem \ref{thm:Z-characterize} and Corollary \ref{cor:complex-T} correspond as in Table \ref{tab:Z-characterize-complex-T-correspond} by Proposition \ref{prop:Kac-Bernstein-Fourier} except for the condition \ref{item:Z-characterize-Gauss} of Theorem \ref{thm:Z-characterize}.
Thus, it suffices to show that either \ref{item:complex-T-Gauss} or \ref{item:complex-T-boundary} of Corollary \ref{cor:complex-T} holds when \ref{item:Z-characterize-Gauss} in Theorem \ref{thm:Z-characterize} holds.
In addition, it suffices to show it when $ ( a_1 , a_2 , k , c_2 , c_4 , c_6 ) = ( 0 , 0 , 1 , 0 , 0 , 0 ) $, that is, when
\begin{align}
    \hat{ \mu_1 } ( u ) = f_1 ( u ) = e^{ c_1 u^2 + c_3 u }, & &
    \hat{ \mu_2 } ( u ) = f_2 ( u ) = e^{ c_1 u^2 + c_5 u } \label{eq:Fourier-simple}
\end{align}
hold for any $ u \in \mathbb{ Z } $ by Lemma \ref{lem:key-complex-T}.
Since $ f_1 $ and $ f_2 $ are bounded by \eqref{eq:Fourier-simple}, we have $ \re c_1 \leq 0 $.

\begin{enumerate}[label=(Case \arabic*)]
    \item 
    
    When $ \re c_1 < 0 $:

    Since
\begin{align}
    \mu_1 = \vartheta \left( \frac{ x - i c_3 }{ 2 \pi } , \frac{ c_1 }{ \pi i } \right) dm_{ \mathbb{ T } } , & &
    \mu_2 = \vartheta \left( \frac{ x - i c_5 }{ 2 \pi } , \frac{ c_1 }{ \pi i } \right) dm_{ \mathbb{ T } }
\end{align}
satisfy \eqref{eq:Fourier-simple}, the condition \ref{item:complex-T-Gauss} of Corollary \ref{cor:complex-T} holds.

    \item 

    When $ \re c_1 = 0 $:

    We define
    \begin{align}
    \theta_1 \coloneqq \frac{ c_3 }{ i } , & & \theta_2 \coloneqq \frac{ c_5 }{ i }
\end{align}
and they are real since $ f_1 $ and $ f_2 $ are bounded.
Let $\rho ( x ) \coloneqq \mu_1 ( x - \theta_1 ) $.
The number $ C \coloneqq c_1 / i $ is real by $ \re c_1 = 0 $, and $ \hat{ \rho } ( u ) = e^{ i C u^2 } $ holds.
Similarly, the Fourier--Stieltjes transform of $ \mu_2 ( x - \theta_2 ) $ is $ e^{ i C u^2 } $ and hence
\begin{align}
    \mu_1 ( x ) = \rho ( x + \theta_1 ), & & 
    \mu_2 ( x ) = \rho ( x + \theta_2 )
\end{align}
hold.
Thus, we obtain \ref{item:complex-T-boundary} of Corollary \ref{cor:complex-T}.
\qedhere
\end{enumerate}
    
\end{proof}

\subsection{Proof of Corollary \ref{cor:Kac-Bernstein-complex-R}}
\label{subsec:Kac-Bernstein-complex-R}

In this subsection, we prove Corollary \ref{cor:Kac-Bernstein-complex-R} by using Corollary \ref{cor:R-characterize} and Proposition \ref{prop:Kac-Bernstein-Fourier}.

\begin{proof}[Proof of Corollary \ref{cor:Kac-Bernstein-complex-R}]

We define the functions $ f_1 $ and $ f_2 $ by \eqref{eq:f1-f2-Fourier}.
Then we have \eqref{eq:complex-Kac-Bernstein} by Proposition \ref{prop:Kac-Bernstein-Fourier} $ \text{ \ref{item:Kac-Bernstein-Fourier-independent} } \Longrightarrow \text{ \ref{item:Kac-Bernstein-Fourier-equation} } $.
Thus, the functions $ f_1 $ and $ f_2 $ satisfy one of the conditions of Corollary \ref{cor:R-characterize}.
The conditions \ref{item:R-characterize-f1-zero} and \ref{item:R-characterize-f2-zero} of Corollary \ref{cor:R-characterize} correspond to \ref{item:Kac-Bernstein-complex-R-f1-zero} and \ref{item:Kac-Bernstein-complex-R-f2-zero} of Corollary \ref{cor:Kac-Bernstein-complex-R}, respectively.
Thus, we may assume the condition \ref{item:R-characterize-Gauss} of Corollary \ref{cor:R-characterize}.
Since $ f_1 $ and $ f_2 $ are bounded by \eqref{eq:f1-f2-Fourier}, we have $ \re d_1 \leq 0 $.

\begin{enumerate}[label=(Case \arabic*)]
\item 
When $ d_1 \neq 0 $:

Take $ c' \in \mathbb{ C } $ satisfying $ e^{ c' } = - \pi / d_1 $.
Then
\begin{align}
\mu_1(x)
& = \exp \left( \frac{ \pi^2 x^2 }{ d_1 } - \frac{ \pi i d_2 x }{ d_1 } + \frac{ c' }{ 2 } + d_3 - \frac{ d_2^2 }{ 4d_1 } \right) dx, \\
\mu_2 ( x )
& = \exp \left( \frac{ \pi^2 x^2 }{ d_1 } - \frac{ \pi i d_4 x }{ d_1 } + \frac{ c' }{ 2 } + d_5 - \frac{ d_4^2 }{ 4d_1 } \right) dx
\end{align}
hold by the Fourier inversion formula (note that these formulas are valid in the sense of Schwartz distributions even when $ \re d_1 = 0 $).
Since $ \mu_1 $ is a complex measure, we have $ \re ( \pi^2 / d_1 ) < 0 $.
Thus, the condition \ref{item:Kac-Bernstein-complex-R-Gauss} of Corollary \ref{cor:Kac-Bernstein-complex-R} holds.

\item
When $ d_1 = 0 $:

Since $ f_1 $ and $ f_2 $ are bounded, the values
\begin{align}
    t_1 \coloneqq \frac{ i d_2 }{ 2 \pi }, & &
    t_2 \coloneqq \frac{ i d_4 }{ 2 \pi }
\end{align}
are real.
Then we have
\begin{align}
    ( e^{ d_3 } \delta_{ t_1 } )\hat{} ( u ) = e^{ d_2 u + d_3 }, & &
    ( e^{ d_5 } \delta_{ t_2 } )\hat{} ( u ) = e^{ d_4 u + d_5 }
\end{align}
and hence 
\begin{align}
    \mu_1 = e^{ d_3 } \delta_{ t_1 }, & &
    \mu_2 = e^{ d_5 } \delta_{ t_2 }.
\end{align}
Thus, the condition \ref{item:Kac-Bernstein-complex-R-degenerate} of Corollary \ref{cor:Kac-Bernstein-complex-R} holds.
\qedhere
\end{enumerate}

\end{proof}

\subsection{Alternative proof of Corollary \ref{cor:Baryshnikov-Eisenberg-Stadje}}
\label{subsec:Baryshnikov-Eisenberg-Stadje}

In this subsection, we obtain an alternative proof of Corollary \ref{cor:Baryshnikov-Eisenberg-Stadje}, which was proved by Baryshnikov--Eisenberg--Stadje.

\begin{proof}[Proof of Corollary \ref{cor:Baryshnikov-Eisenberg-Stadje}]

Defining $ f_1 $ and $ f_2 $ as \eqref{eq:f1-f2-Fourier}, we have \eqref{eq:Kac-Bernstein-original} by Corollary \ref{cor:original-Kac-Bernstein-Fourier} $ \text{ \ref{item:original-Kac-Bernstein-Fourier-independent} } \Longrightarrow \text{ \ref{item:original-Kac-Bernstein-Fourier-equation} } $.
In addition, we also have \eqref{eq:probability-dual} and hence the equivalent conditions of Proposition \ref{prop:original-Kac-Bernstein} hold.
Thus, at least one of the conditions of Corollary \ref{cor:original-Kac-Bernstein-Z} holds.

The conditions \ref{item:original-Kac-Bernstein-Z-f1-one-point} and \ref{item:original-Kac-Bernstein-Z-f2-one-point} of Corollary \ref{cor:original-Kac-Bernstein-Z} correspond to \ref{item:Baryshnikov-Eisenberg-Stadje-f1-one-point} and \ref{item:Baryshnikov-Eisenberg-Stadje-f2-one-point} of Corollary \ref{cor:Baryshnikov-Eisenberg-Stadje}, respectively.
Since $ \mu_1 $ and $ \mu_2 $ are probability measures, their Fourier--Stieltjes transforms $ f_1 $ and $ f_2 $ satisfy \eqref{eq:Hermitian-nonvanish} by \eqref{eq:f1-f2-Fourier}.
Thus, the conditions \ref{item:original-Kac-Bernstein-Z-six-point-f1-symmetry} and \ref{item:original-Kac-Bernstein-Z-six-point-f2-symmetry} of Corollary \ref{cor:original-Kac-Bernstein-Z} do not hold.
Therefore, it suffices to show the conclusion when \ref{item:original-Kac-Bernstein-Z-Gauss} of Corollary \ref{cor:original-Kac-Bernstein-Z} holds.
In addition, it suffices to consider the case of $ k = 1 $.
Then there exist $ q_3 , q_6 \in \mathbb{ C } $ such that
\begin{align}
    f_1 ( 1 ) = e^{ q_3 }, & & f_2 ( 1 ) = e^{ q_6 }
\end{align}
by $ f_1 ( 1 ), f_2 ( 1 ) \neq 0 $.
Since
\begin{align}
    f_1 ( - 1 ) = \overline{ f_1 ( 1 ) } = \overline{ e^{ q_3 } } = e^{ \overline{ q_3 } } , & &
    f_2 ( - 1 ) = \overline{ f_2 ( 1 ) } = \overline{ e^{ q_6 } } = e^{ \overline{ q_6 } }
\end{align}
hold by \eqref{eq:Hermitian-nonvanish}, we can replace $ c_1 $, $ c_2 $, $ c_3 $, and $ c_5 $ with
\begin{align}
    \begin{pmatrix}
        c_1 \\
        c_2 \\
        c_3 \\
        c_5
    \end{pmatrix}
    = \frac{ 1 }{ 8 }
    \begin{pmatrix}
        2 & 2 & 2 & 2 \\
        - 1 & - 1 & 1 & 1 \\
        - 4 & 4 & 0 & 0 \\
        0 & 0 & - 4 & 4
    \end{pmatrix}
    \begin{pmatrix}
        \overline{ q_3 } \\
        q_3 \\
        \overline{ q_6 } \\
        q_6
    \end{pmatrix}
    = 
    \begin{pmatrix}
        \re ( q_3 + q_6 ) / 2 \\
        \re ( q_6 - q_3 ) / 4 \\
        i \im q_3 \\
        i \im q_6
    \end{pmatrix}
    \label{eq:Baryshnikov-Eisenberg-Stadje-proof-const}
\end{align}
by Lemma \ref{lem:supp-all} $ \text{ \ref{item:supp-all-Kac-Bernstein} } \Longrightarrow \text{ \ref{item:supp-all-Gauss} } $.
Then $ c_1 , c_2 \in \mathbb{ R } $ and
\begin{align}
\theta_1 \coloneqq \frac{ c_3 }{ i } = \im q_3 \in \mathbb{ R }, & &
\theta_2 \coloneqq \frac{ c_5 }{ i } = \im q_6 \in \mathbb{ R }
\end{align}
hold.
Since $ f_1 $ and $ f_2 $ are bounded by \eqref{eq:f1-f2-Fourier}, we have $ c_1 \leq 0 $ by \eqref{eq:original-Kac-Bernstein-Z-Gauss-represent}.

\begin{enumerate}[label=(Case \arabic*)]
\item 

When $ c_1 < 0 $:

It suffices to show \ref{item:Baryshnikov-Eisenberg-Stadje-Gauss} when $ c_2 = 0 $, that is, the equations \eqref{eq:Fourier-simple} hold.
We have $ r \coloneqq - c_1 / \pi > 0 $ by $ c_1 < 0 $.
Then
\begin{align}
    \mu_1 \coloneqq \vartheta \left( \frac{ x + \theta_1 }{ 2 \pi } , i r \right) dm_{ \mathbb{ T } } ( x ) , & &
    \mu_2 \coloneqq \vartheta \left( \frac{ x + \theta_2 }{ 2 \pi } , i r \right) dm_{ \mathbb{ T } } ( x )
\end{align}
satisfy \eqref{eq:Fourier-simple} and hence we obtain \ref{item:Baryshnikov-Eisenberg-Stadje-Gauss}.

\item

When $ c_1 = 0 $:

We have
\begin{align}
    \mu_1 = e^{ - c_2 } ( \cosh ( c_2 ) \delta_{ \theta_1 } + \sinh ( c_2 ) \delta_{ \theta_1 + \pi } ) , & &
    \mu_2 = e^{ c_2 } ( \cosh ( c_2 ) \delta_{ \theta_2 } - \sinh ( c_2 ) \delta_{ \theta_2 + \pi } )
\end{align}
by Lemma \ref{lem:key-complex-T}.
Since $ \mu_1 $ and $ \mu_2 $ are probability measures, we have $ c_2 = 0 $.
Thus, we get
\begin{align}
    \mu_1 = \delta_{ \theta_1 } , & &
    \mu_2 = \delta_{ \theta_2 }
\end{align}
and hence \ref{item:Baryshnikov-Eisenberg-Stadje-Zk} holds.
\qedhere
\end{enumerate}
    
\end{proof}

\section*{Acknowledgment}
The author would like to thank Toshiyuki Kobayashi, Yuichiro Tanaka, and Toshihisa Kubo and the anonymous referees
for their careful comments.

\printbibliography

\noindent
Takashi Satomi: 

\noindent
School of System Design and Technology, Department of Mathematics and Data Science, Tokyo Denki University, Adachi-ku, Tokyo 120-8551, Japan.

\noindent
RIKEN Interdisciplinary Theoretical and Mathematical Sciences (iTHEMS), Wako, Saitama 351-0198, Japan.

\noindent
E-mail: takashi.satomi@mail.dendai.ac.jp

\end{document}